\documentclass[reqno]{amsart}
\usepackage{amssymb}
\usepackage[mathscr]{euscript}
\usepackage{amsmath}
\makeatletter
\@addtoreset{equation}{section}
\makeatother

\renewcommand\thefigure{\thesection.\@arabic\c@figure}
\renewcommand\thetable{\thesection.\@arabic\c@table}

\newtheorem{theorem}{Theorem}[section]
\newtheorem{lemma}[theorem]{Lemma}
\newtheorem{proposition}[theorem]{Proposition}

\newtheorem{remark}[theorem]{Remark}

\newcommand{\mc}[1]{{\mathcal #1}}
\newcommand{\ms}[1]{{\mathscr #1}}
\newcommand{\mf}[1]{{\mathfrak #1}}
\newcommand{\mb}[1]{{\mathbf #1}}
\newcommand{\bb}[1]{{\mathbb #1}}
\newcommand{\bs}[1]{{\boldsymbol #1}}
\newcommand{\mce}[1]{{\mathscr E_N^{#1}}}

\newcommand{\<}{\langle}
\renewcommand{\>}{\rangle}

\renewcommand{\Cap}{{\rm cap}}

\begin{document}

\title[Metastability in asymmetric zero range processes]{Metastability
  for a non-reversible dynamics: the evolution of the condensate in
  totally asymmetric zero range processes}

\author{C. Landim}

\address{\noindent IMPA, Estrada Dona Castorina 110, CEP 22460 Rio de
  Janeiro, Brasil and CNRS UMR 6085, Universit\'e de Rouen, Avenue de
  l'Universit\'e, BP.12, Technop\^ole du Madril\-let, F76801
  Saint-\'Etienne-du-Rouvray, France.  \newline e-mail: \rm
  \texttt{landim@impa.br} }


\begin{abstract}
  Let $\bb T_L = \bb Z/L \bb Z$ be the one-dimensional torus with $L$
  points. For $\alpha >0$, let $g: \bb N\to \bb R_+$ be given by
  $g(0)=0$, $g(1)=1$, $g(k) = [k/(k-1)]^\alpha$, $k\ge 2$. Consider
  the totally asymmetric zero range process on $\bb T_L$ in which a
  particle jumps from a site $x$, occupied by $k$ particles, to the
  site $x+1$ at rate $g(k)$. Let $N$ stand for the total number of
  particles. In the stationary state, if $\alpha >1$, as
  $N\uparrow\infty$, all particles but a finite number accumulate on
  one single site. We show in this article that in the time scale
  $N^{1+\alpha}$ the site which concentrates almost all particles
  evolves as a random walk on $\bb T_L$ whose transition rates are
  proportional to the capacities of the underlying random walk,
  extending to the asymmetric case the results obtained in \cite{bl3}
  for reversible zero-range processes on finite sets.
\end{abstract}

\maketitle

\section{Introduction}
\label{sec-1} 

Metastability is a relevant dynamical phenomenon in the framework of
non-equilibrium statistical mechanics, which occur in the vicinities
of first order phase transitions. We refer to the monograph \cite{ov}
for an overview of the literature.

Recently, \cite{bl2} after \cite{begk1, begk2, g} proposed a new
approach to metastability for reversible dynamics based on potential
theory. They applied this method in \cite{bl3} to prove the metastable
behavior of the condensate in sticky reversible zero range processes
evolving on finite sets and in \cite{bl4, bl5} to examine the
metastability of reversible Markov processes evolving on fixed finite
sets. These methods were also used in \cite{jlt1, jlt2} to investigate
the scaling limits of trap models.

More recently, we extended in \cite{gl2} the potential theory of
reversible dynamics to the non-reversible context by proving a
Dirichlet principle for Markov chains on countable state spaces. In
contrast with the reversible case, the formula for the capacity
involves a double variational problem, and it wasn't clear from this
additional difficulty if such principle could be of any utility.

In this article, we use this Dirichlet principle for non-reversible
dynamics to prove the metastable behavior of the condensate in
sticky totally asymmetric zero range processes evolving on a
fixed one-dimensional torus. This is, to our knowledge, the first
proof of a metastable behavior of a non-reversible dynamics.

The first main message we want to convey is that the variational
formula \eqref{09} for the capacity between two sets for
non-reversible dynamics should be understood as an infimum over
functions $H$ which satisfy certain boundary conditions and which
solve the equation $\mc SH = \mc L^*F$ for functions $F$ which satisfy
similar boundary conditions. Here, $\mc L^*$ stands for the adjoint of
the generator of the Markov process and $\mc S$ for its symmetric
part.  In this sense, it is similar to the known variational formula
for reversible processes and one can use similar techniques to
estimate the capacities. We illustrate this assertion by examining the
metastable behavior of the condensate for asymmetric zero range
dynamics.

\smallskip\noindent{\bf Condensation.} We conclude this introduction
with a few words on condensation.  The stationary states of
sticky zero range processes exhibit a very peculiar structure
called condensation in the physics literature.  Mathematically, this
means that under the stationary state, above a certain critical
density a macroscopic number of particles concentrate on a single site
\cite{jmp, gss, gl, emz06, fls, al, al2, agl}.  This phenomenon has
been observed and investigated in shaken granular systems, growing and
rewiring networks, traffic flows and wealth condensation in
macroeconomics \cite{eh}.

Once the presence of a condensate at the stationary state has been
established, one is tempted to investigate its time evolution. This
has been done in \cite{bl3} for reversible dynamics, where the authors
prove that on a certain time scale the position of the condensate
evolves as a random walk with jump rates proportional to the
capacities of the underlying randoms walks. This surprising fact is
also observed in the asymmetric regime as shown in Theorem \ref{mt2}
below.

\section{Notation and results}
\label{sec0}

Denote by $\bb T_L$ the one dimensional discrete torus with $L$ sites
and let $E = \bb N^{\bb T_L}$ be the set of configurations on $\bb
T_L$. The configurations are denoted by the Greek letters $\eta$ and
$\xi$. In particular, $\eta_x$, $x\in\bb T_L$, represents the number
of particles at site $x$ for the configuration $\eta$.

Fix a real number $\alpha>0$, define $a(n)=n^{\alpha}$, $n\ge 1$, and
set $a(0)=1$. Let us also define $g:\bb N\to \bb R_+$,
$$
g(0)=0 \quad
\textrm{and}\quad g(n)=\frac{a(n)}{a(n-1)}\;, \;\; n\ge 1\;,
$$
in such a way that $\prod_{i=1}^{n}g(i)=a(n)$, $n\ge 1$, and that $\{
g(n) : n\ge 2\}$ is a strictly decreasing sequence converging to $1$
as $n\uparrow\infty$.

For each pair of sites $x$, $y\in \bb T_L$ and configuration $\eta\in
E$ such that $\eta_x>0$, denote by $\sigma^{x, y}\eta$ the
configuration obtained from $\eta$ by moving a particle from $x$ to
$y$:
$$
(\sigma^{x,y}\eta)_z\;=\;\left\{
\begin{array}{ll}
\eta_x-1 & \textrm{for $z=x$} \\
\eta_{y}+1 & \textrm{for $z=y$} \\
\eta_z & \rm{otherwise}\;. \\
\end{array}
\right.
$$
Denote by $\{\eta (t) : t\ge 0\}$ the Markov process on $E$
whose generator $\mc L$ acts on functions $F: E\to\bb R$ as 
\begin{equation}
\label{f16}
(\mc L F) (\eta) \;=\; \sum_{x\in\bb T_L} 
g(\eta_x) \, \big\{ F(\sigma^{x, x+1}\eta) - F(\eta) \big\} \;. 
\end{equation}
This process is known as the totally asymmetric zero range process
with jump rate $g(\cdot)$.

\smallskip\noindent{\bf First order phase transition.} Let
$Z(\varphi)$, $\varphi>0$, be the partition function
\begin{equation*}
Z(\varphi) \;=\; \sum_{n\ge 0} \frac {\varphi^n}{a(n)} \; \cdot 
\end{equation*}
For $\alpha>0$, the radius of convergence of this series is clearly
equal to $1$.  A simple computation shows that for each $\varphi<1$,
$\varphi\le 1$ if $\alpha >1$, the product measure $\nu_\varphi$ on
$E$ with marginals given by
\begin{equation*}
\nu_\varphi\{\eta: \eta (x)=k\} \;=\; \frac 1{Z(\varphi)}
\frac{\varphi^k}{a(k)} \;,\quad x\in\bb T_L\;, \;k\ge 0\;, 
\end{equation*}
is a stationary measure.

Denote by $R(\varphi)$ the average density of particles under the
measure $\nu_\varphi$:
\begin{equation*}
R(\varphi) \;:=\; E_{\nu_\varphi} \big[\eta_0\big] \;=\;
\frac{\varphi \, Z'(\varphi)}{Z(\varphi)}\;\cdot
\end{equation*}
It is easy to show that $R(0) =0$ and that $R$ is strictly increasing
since $R'(\varphi) = \varphi^{-1} {\rm Var}_{\nu_\varphi}[\eta_0]$.
There are three different regimes. For $\alpha \le 1$, $Z(\varphi)$
increases to $\infty$ as $\varphi$ converges $1$. In particular, for
each density $\rho\in [0,\infty)$, there exists a stationary measure
$\nu_\varphi$ whose average density is $\rho$. For $1<\alpha\le 2$,
$Z(\varphi)$ increases to $Z(1)<\infty$ as $\varphi$ converges $1$,
but $Z'(\varphi)$ increases to $\infty$ as $\varphi\uparrow 1$.  In
this case also for each density $\rho\in [0,\infty)$, there exists a
stationary measure $\nu_\varphi$ whose average density is $\rho$. In
contrast, for $\alpha> 2$, $Z(\varphi)$ and $Z'(\varphi)$ converge to
finite values as $\varphi\uparrow 1$, and we have a phase
transition. Only for densities $\rho$ in the interval $[0,R(1)]$ there
are stationary measures $\nu_\varphi$ with average density $\rho$. In
fact, in \cite{fls} we proved that for fixed $L$ and for $\alpha>2$,
if we denote by $N$ the total number of particles and if we let
$N\uparrow\infty$, all but a finite number of particles concentrate on
one site, a phenomena called condensation and observed also in the
thermodynamical limit as $L\uparrow \infty$ together with $N$ in such
a way that the density $N/L$ converges to $\rho>R(1)$, \cite{jmp, gss,
  gl, emz06, al}.

\smallskip\noindent{\bf Stationary states.}
For $N\ge 1$, denote by $E_N$ the set of configurations with $N$
particles:
\begin{equation*}
E_N \;=\; \big\{ \eta\in E : \sum_{x\in\bb T_L} \eta_x = N \big\}\;.
\end{equation*}
Since the dynamics conserves the total number of particles, the sets
$E_N$, $N\ge 1$, are the irreducible classes of the Markov process
$\eta(t)$. It will be convenient to represent the zero-range process
on $E_N$ as a random walk on the simplex $\{(i_1, \dots , i_{L-1}) :
i_k \ge 0\, ,\, i_1 + \dots + i_{L-1} \le N\}$.

Let $\mu_N$ be the probability measure on $E_N$ obtained from
$\nu_\varphi$ by conditioning on the total number of particles being
equal to $N$: $\mu_N(\eta) = \nu_\varphi ( \eta | \sum_{0\le x<L}
\eta_x = N)$. The measure $\mu_N$ does not depend on the parameter
$\varphi$ and a calculation shows that
$$
\mu_N(\eta) \;=\; \frac{N^{\alpha}}{Z_{N}} \, \frac{1}
{a(\eta)} \;:=\;  \frac {N^{\alpha}} {Z_{N}} \, 
\prod_{x\in \bb T_L} \frac{1}{ a(\eta_x)} \;, \quad \eta\in E_N \;,
$$
where $Z_{N}$ is the normalizing constant
\begin{equation}
\label{f03}
Z_{N}\;=\; N^{\alpha} 
\sum_{\zeta\in E_{N}} \frac{1}{a(\zeta)}\;\cdot
\end{equation} 

By \cite[Proposition 2.1]{bl3}, 
\begin{equation}
\label{zk}
\lim_{N\to\infty} Z_{N} \;=\;
L \, \Gamma(\alpha)^{L -1}\;, \quad
\text{where}\quad \Gamma(\alpha) \;:=\; \sum_{j\ge 0} \frac{1}{a(j)} \;\cdot
\end{equation}

An elementary computation shows that $\mu_N$ is the stationary state
of the zero range process with generator $\mc L$ restricted to
$E_N$. More precisely, let $\mc L^*$ be the adjoint of the generator
$\mc L$ in $L^2(\mu_N)$. On can check that $\mc L^*$ is the generator
of the totally asymmetric zero range process in which particles jump to
the left instead of jumping to the right:
\begin{equation*}
(\mc L^* F) (\eta) \;=\; \sum_{x\in\bb T_L} 
g(\eta_x) \, \big\{ F(\sigma^{x, x-1}\eta) - F(\eta) \big\} \;. 
\end{equation*}
Denote by $\<\,\cdot\,,\,\cdot\,\>_{\mu_N}$ the scalar product
in $L^2(\mu_N)$. A change of variables gives that
\begin{equation*}
\< \mc L F \,,\, G\>_{\mu_N} \;=\; \< F \,,\, \mc L^* G\>_{\mu_N}
\end{equation*}
for every function $F$, $G:E_N\to \bb R$. In particular, taking $G=1$,
as $\mc L^*1=0$, $\mu_N$ is the stationary state for the process
restricted to $E_N$.

\smallskip\noindent{\bf Capacities.}
Denote by $\{\eta^*(t) : t\ge 0\}$ the Markov process on $E$ whose
generator is $\mc L^*$. We shall refer to $\eta^*(t)$ as the adjoint
or the time reversed process. 

For a subset $\ms A$ of $E$, denote by $H_{\ms A}$ (resp. $H^+_{\ms
  A}$) the hitting (resp. return) time of a set $\ms A$:
\begin{equation*}
\begin{split}
& H_{\ms A} \,:=\, \inf \big\{ s > 0 : \eta(s) \in \ms A \big\}\;, \\
& \quad H^+_A \,:=\, \inf \{ t>0 : \eta(t) \in \ms A\,, \eta(s)\not= \eta(0)
\;\;\textrm{for some $0< s < t$}\}\;.
\end{split}
\end{equation*}
When the set $\ms A$ is a singleton $\{\xi\}$, we denote $H_{\{\xi\}}$,
$H^+_{\{\xi\}}$ by $H_\xi$, $H^+_{\xi}$, respectively.

For each $\eta\in E$, let ${\bf P}_{\eta}$ stand for the probability
on the path space of right continuous trajectories with left limits,
$D(\bb R_+, E)$, induced by the zero range process $\{\eta(t) : t\ge
0\}$ starting from $\eta\in E$. Expectation with respect to ${\bf
  P}_{\eta}$ is denoted by ${\bf E}_{\eta}$. Similarly, we denote by
${\bf P}^*_{\eta}$, ${\bf E}^*_{\eta}$ the probability and the
expectation on $D(\bb R_+, E)$ induced by the time reversed process
$\{\eta^*(t) : t\ge 0\}$ starting from $\eta\in E$.

For two disjoint subsets $\ms A$, $\ms B$ of $E_N$, denote by $V_{\ms
  A,\ms B}$, $V^*_{\ms A,\ms B}$ the equilibrium potentials defined by
\begin{equation}
\label{07}
V_{\ms A,\ms B}(\eta) \;=\; \mb P_\eta[H_{\ms A} < H_{\ms B}]\;, \quad
V^*_{\ms A,\ms B}(\eta) \;=\; \mb P^*_\eta[H_{\ms A} < H_{\ms B}] \;, \quad
\eta\in E_N\;.
\end{equation}

Denote by $\mc S$ the symmetric part of the generator $\mc L$: $\mc S =
(1/2)(\mc L+ \mc L^*)$, and by $D_N$ the Dirichlet form associated to the
generator $\mc L$. An elementary computation shows that
\begin{equation}
\label{11}
(\mc S F) (\eta) \;=\; \frac 12 \sum_{x\in\bb T_L} \sum_{y=-1,1}
g(\eta_x) \, \big\{ F(\sigma^{x, x+y}\eta) - F(\eta) \big\} 
\end{equation}
and that
\begin{equation*}
D_N(F)\;=\; \< F \,,\, (-\mc S) F\>_{\mu_N} \;=\; 
\frac 12 \sum_{x\in \bb T_L} \sum_{\eta\in E_N} \mu_N(\eta)\, 
g(\eta_x) \, \{ F(\sigma^{x, x+1}\eta) - F(\eta) \}^2 \;,
\end{equation*}
for every $F:E_N\to\bb R$.

For two disjoint subsets $\ms A$, $\ms B$ of $E_N$, let $\mc C (\ms A,
\ms B)$ be the set of functions $h: E_N \to \bb R$ which are constant
over $A$ and constant over $\ms B$, with possibly different values at
$\ms A$ and $\ms B$. Let $\mc C_{1,0} (\ms A, \ms B)$ be the subset of
functions in $\mc C (\ms A, \ms B)$ equal to $1$ on $\ms A$ and $0$ on
$\ms B$.

Let $\Cap_N(\ms A,\ms B)$ be the capacity between two disjoint subsets
$\ms A$, $\ms B$ of $E_N$, defined in \cite[Definition 1.1]{gl2}, and
recall from \cite[Theorem 1.4]{gl2} the variational formula for the
capacity:
\begin{equation}
\label{09}
\Cap_N (\ms A,\ms B)\,=\, \inf_F \, \sup_H 
\Big\{ 2 \<  \mc L^* F \,,\, H\>_{\mu_N}  \, - 
\, \< H , (- \mc S) H\>_{\mu_N} \Big\} \;,
\end{equation}
where the supremum is carried over all functions $H$ in $\mc C (\ms A,
\ms B)$, and where the infimum is carried over all functions $F$ in
$\mc C_{1,0} (\ms A, \ms B)$. When the set $\ms A$ is a singleton,
$A=\{\xi\}$, we denote the capacity $\Cap_N (\ms A,\ms B)$ by $\Cap_N
(\xi,\ms B)$.

We have shown in \cite{gl2} that the function $F_{\ms A,\ms B}$ which
solves the variational problem for the capacity is equal to $(1/2) \{
V_{\ms A,\ms B} + V^*_{\ms A,\ms B}\}$, where $V_{\ms A,\ms B}$,
$V^*_{\ms A,\ms B}$ are the harmonic functions defined in \eqref{07},
and that $\Cap_N(\ms A,\ms B) = D(V_{\ms A,\ms B})$.

Consider the continuous time totally asymmetric random walk $\{ X(t)
\,|\, t\ge 0\}$ on $\bb T_L$ jumping to the right with rate one. The
stationary measure is the uniform measure. Denote by $\Cap (A,B)$ the
capacity between two disjoint sets $A$, $B$ of $\bb T_L$. One can
compute the capacity between two sites $x\not = y\in\bb T_L$ recalling
the observation made in the previous paragraph. Clearly, $V_{x,y}$ is
the indicator of the set $\{y+1, \dots, x\}$ and $V^*_{x,y}$ is the
indicator of the set $\{x, \dots, y-1\}$. Hence, the solution
$F_{x,y}$ of the variational problem \eqref{09} is given by $F_{x,y}
(z) = \delta_{x,z} + (1/2) \mb 1\{ z \not \in \{x,y\}\, \}$, and
$\Cap(x,y) = D(V_{x,y})=L^{-1}$ is independent of $x$, $y$.

\smallskip\noindent{\bf Tunneling.}
Fix a sequence $\{\ell_N : N\ge 1\}$ such that $1\ll \ell_N \ll N$:
\begin{equation}
\label{f18}
\lim_{N\to\infty} \ell_N \;=\; \infty
\quad\textrm{and}\quad
\lim_{N\to\infty} \ell_N/N \;=\; 0\;.
\end{equation}
For $x$ in $\bb T_L$, let
\begin{equation*}
\ms E^x_N  \;:=\; \Big\{\eta\in E_N : \eta_x \ge N - \ell_N  \Big\}\;.
\end{equation*}
Obviously, $\ms E^x_N\not = \varnothing$ for all $x\in \bb T_L$ and
every $N$ large enough. 

Condition $\ell_N/N\to 0$ is required to guarantee that on each set
$\ms E^x_N$ the proportion of particles at $x\in \bb T_L$,
i.e. $\eta_x/N$, is almost one. As a consequence, for $N$ sufficiently
large, the subsets $\ms E^x_N$, $x\in \bb T_L$, are pairwise disjoint.
 From now on, we assume that $N$ is large enough so that the partition
\begin{equation}
\label{f15}
E_N\;=\; \ms E_N \cup \Delta_N \;:=\;
\Big( \bigcup_{x\in \bb T_L}\mce x \Big) \cup \Delta_N
\end{equation}
is well defined, where $\Delta_N$ is the set of configurations which
do not belong to any set $\ms E^x_N$, $x\in \bb T_L$.  

The assumptions that $\ell_N\uparrow \infty$ are sufficient to prove
that $\mu_N(\Delta_N)\to 0$, as we shall see in Section \ref{sec5},
and to deduce the limit of the capacities stated in Theorem \ref{mt1}
below. We shall need, however, further restrictions on the growth of
$\ell_N$ to prove the tunneling behaviour of the zero range processes
presented in Theorem \ref{mt2} below.

To state the first main result of this article, for any nonempty
subset $A$ of $\bb T_L$, let $\mce{}(A)\,=\, \cup_{x\in A} \mce x$,
and let
\begin{equation}
\label{defi}
I_\alpha \;:=\; \int_{0}^1 u^\alpha(1-u)^\alpha \, du\;.
\end{equation}

\begin{theorem}
\label{mt1}
Assume that $\alpha>3$ and consider a sequence $\{\ell_N : N\ge 1\}$
satisfying \eqref{f18}. Then, for all proper subset $A$ of $\bb T_L$,
\begin{equation*}
\lim_{N\to\infty} N^{1+\alpha} \Cap_N\big(\ms E_N(A), \ms
E_N(A^c)\big) \;=\; \frac 1{\Gamma(\alpha) \, I_\alpha} 
\sum_{x \in A, y\not\in A} \Cap(x,y) \; .
\end{equation*}
\end{theorem}

We have seen above that $\Cap (x,y) = L^{-1}$ for all $x$, $y\in\bb
T_L$, $x\not = y$. Therefore, under the assumption of the previous
theorem,
\begin{equation*}
\lim_{N\to\infty} N^{1+\alpha} \Cap_N\big(\ms E_N(A), \ms
E_N(A^c)\big) \;=\; \frac 1{\Gamma(\alpha) \, I_\alpha} 
\,\frac {|A| (L-|A|)}{L} \; ,
\end{equation*}
where $|A|$ stands for the cardinality of the set $A$. \medskip

The second main result of this article states that the zero range
process exhibits a metastable behavior.  Fix a nonempty subset $\ms A$
of $E_N$. For each $t\ge 0$, let $\mc T^{\ms A}_t$ be the time spent
by the zero range process $\{\eta(t) : t\ge 0\}$ on the set $\ms A$ in
the time interval $[0,t]$:
$$
\mc T^{\ms A}_t \;:=\;\int_{0}^t \mathbf{1}\{\eta(s) \in \ms A\} \,ds\;,
$$
and let $\mc S^{\ms A}_t$ be the generalized inverse of $\mc T^{\ms A}_t:$
$$
\mc S^{\ms A}_t \;:=\;\sup\{s\ge 0 : \mc T^{\ms A}_s \le t\}\;.
$$
It is well known that the process $\{\eta^{\ms A}(t) : t\ge 0\}$ defined
by $\eta^{\ms A}(t) = \eta({\mc S^{\ms A}_t})$ is a strong Markov process
with state space $\ms A$ \cite{bl2}. This Markov process is called
the trace of the Markov process $\{\eta(t) : t\ge 0\}$ on $\ms A$.

Consider the trace of $\{\eta(t) : t\ge 0\}$ on $\ms E_N$, referred to
as $\eta^{\ms E_N}(t)$. Let $\Psi_N:\ms E_N\mapsto \bb T_L$ be given by
$$
\Psi_N(\eta) \;=\; \sum_{x\in \bb T_L} x\, \mathbf 1\{\eta \in \ms
E^x_N\} 
$$
and let $X^N_t:=\Psi_N(\eta^{\ms E_N}(t))$.

We prove in Theorem \ref{mt2} below that the speeded up non-Markovian
process $\{X^N_{tN^{\alpha+1}} : t\ge 0\}$ converges to the random
walk $\{X_t : t\ge 0\}$ on $\bb T_L$ whose generator $\mf L$
is given by
\begin{equation}
\label{f17}
(\mf L f) (x) \;=\;  \frac L 
{\Gamma(\alpha) \, I_{\alpha} } 
\sum_{y\in \bb T_L} \Cap (x,y)\, \{f(y) - f(x) \}
\;=\;  \frac 1 {\Gamma(\alpha) \, I_{\alpha} } 
\sum_{y\in \bb T_L}  \{f(y) - f(x) \}\;.
\end{equation}
For $x$ in $\bb T_L$, denote by $\bb P_x$ the probability measure on
the path space $D(\bb R_+, \bb T_L)$ induced by the random walk $\{X_t
: t\ge 0\}$ starting from $x$.

\renewcommand{\theenumi}{\Alph{enumi}}
\renewcommand{\labelenumi}{(\theenumi)}

\begin{theorem}
\label{mt2}
Assume that $\alpha > 3$ and that $1 \ll \ell_N \ll N^\gamma$, where
$\gamma = (1+\alpha)/[ 1+ \alpha(L-1)]$.  Then, for each $x\in \bb
T_L$,
\begin{enumerate}
\item[({\bf M1})]
We have
\begin{equation*}
\lim_{N\to\infty} \inf_{\eta,\xi\in \ms E^x_N} {\bf P}^N_{\eta}
\big[\,H_{\{ \xi\}} < H_{{\ms E}_N(\bb T_L\setminus \{x\})}\,\big] \;=\; 1 \;;
\end{equation*}

\item[({\bf M2})] For any sequence $\xi_N\in \ms E^x_N$, $N\ge 1$, the
  law of the stochastic process $\{X^N_{tN^{\alpha +1}} : t\ge 0\}$
  under ${\bf P}^N_{\xi_N}$ converges to $\bb P_{x}$ as
  $N\uparrow\infty$;

\item[({\bf M3})]
For every $T> 0$,
$$
\lim_{N\to\infty} \sup_{\eta\in \ms E^x_N} {\bf E}^N_{\eta} \Big[\, 
\int_0^T {\bs 1}\big\{ \eta(sN^{\alpha +1})\in \Delta_N \big\}\, ds \,\Big]\;
=\; 0 \;.
$$
\end{enumerate}
\end{theorem}

The assumption that $\ell_N \ll N^\gamma$ is needed to prove
assumption ({\bf H1}) of metastability stated in Section
\ref{sec5}. It should be possible to relax the assumption that
$\alpha>3$ if one tackles carefully Step 3 of the proof of Proposition
\ref{s05}, but our purpouse here is not to give the optimal conditions
for Theorem \ref{mt2}.  Our main point is to show how to estimate
capacities in the non-reversible case where these capacities are given
by a double variational formula. We claim that the variational problem
appearing in the definition \eqref{09} of the capacity $\Cap (\ms A,
\ms B)$ has to be understood as the variational problem
\begin{equation*}
\inf_H D(H) \;,
\end{equation*}
where the infimum is performed over functions $H$ in $\mc C(\ms A, \ms
B)$ such that $\mc SH = - \mc L^*F$ for some function $F$ in $\mc
C_{1,0}(\ms A, \ms B)$. We hope that the proof of Proposition
\ref{s05} will clarify this affirmation and will convince the reader
of its correctness.

For the same reasons, we concentrated on the totally asymmetric case,
where the computations are simpler. An analogous result should hold
for asymmetric dynamics since the main tool pervading all the argument
is a sector condition which holds in all asymmetric cases.

According to the terminology introduced in \cite{bl2}, Theorem
\ref{mt2} states that the sequence of zero range processes
$\{\eta (t) : t\ge 0\}$ exhibits a tunneling behaviour on the
time-scale $N^{\alpha + 1}$ with metastates given by $\{\ms E^x_N :
x\in \bb T_L \}$ and limit given by the random walk $\{X_t : t\ge
0\}$.

The asymptotic evolution of the condensate is reversible even though
the original dynamics is not. It does not coincide, however, with the
asymptotic dynamics of the condensate in the reversible case where
particles jump to the left and to the right neighbors with equal
probability $1/2$ \cite{bl3}.

Property ({\bf M3}) states that, outside a time set of order smaller
than $N^{\alpha +1}$, one of the sites in $\bb T_L$ is occupied by at
least $N-\ell_N$ particles. Property ({\bf M2}) describes the
time-evolution on the scale $N^{\alpha +1}$ of the condensate. It
evolves asymptotically as a Markov process on $\bb T_L$ which jumps
from a site $x$ to $y$ at a rate proportional to the capacity
$\Cap(x,y)$ of the underlying random walk. Property ({\bf M1})
guarantees that the process starting in a metastate $\ms E^x_N$
thermalizes therein before reaching any other metastate.

\section{Sector condition}
\label{sec2}

It has been proved in \cite{gl2} that we may estimate the capacity of
a non-reversible process with the capacity of the reversible version
of the process if a sector condition is in force. The first result of
this section establishes a sector condition for the totally asymmetric
zero range process.

\begin{lemma}
\label{u02}
The zero range process with generator $\mc L$ defined in \eqref{f16}
satisfies a sector condition with constant $4 L^2$: For every pair of
functions $F$, $H: E_{N}\to \bb R$,
\begin{equation}
\label{21}
\<\mc L F , H\>_{\mu_N}^2 \;\le\; 4\, L^2 \, D_N(F) \, D_N(H)\;. 
\end{equation}
\end{lemma}

\begin{proof}
Denote by $\mf d_z$, $z\in \bb T_L$, the configuration of $E_1$ with
one particle at $z\in \bb T_L$, where the sum of two configurations
$\eta$, $\xi$ is performed by summing each component: $(\eta +
\xi)(x)=\eta(x) + \xi(x)$.

Fix two functions $F$, $H: E_{N}\to \bb R$. By the the change of
variables $\xi = \eta -\mf d_x$, $\<\mc L F , H\>_{\mu_N}$ can be
written as
\begin{equation*}
\frac{W_{N-1}}{W_{N}} \sum_{\xi\in E_{N-1}} \mu_{N-1}(\xi) \sum_{x\in
  \bb T_L} [F(\xi + \mf d_{x+1}) - F(\xi + \mf d_x)]\,
H(\xi + \mf d_x)\;,
\end{equation*}
where $W_N = Z_N/N^\alpha$.  Fix $\xi\in E_{N-1}$ and consider the sum
\begin{equation*}
\sum_{x \in \bb T_L}  
[F(\xi + \mf d_{x+1}) - F(\xi + \mf d_x)]\, H(\xi + \mf d_x)\;,
\end{equation*}
which can be rewritten as
\begin{equation*}
\sum_{x \in \bb T_L} [F(\xi + \mf d_{x+1}) - F(\xi + \mf d_x)]\, \Big\{
H(\xi + \mf d_x) - \frac 1 L \sum_{z\in \bb T_L} H(\xi + \mf d_z)
\Big\} \;.
\end{equation*}
Since $2ab \le \gamma a^2 + \gamma^{-1}b^2$, $\gamma>0$, by Schwarz
inequality, this expression is less than or equal to
\begin{equation*}
\begin{split}
& \frac \gamma 2 \sum_{x \in \bb T_L} 
[F(\xi + \mf d_{x+1}) - F(\xi + \mf d_x)]^2\, \\
& \quad +\; \frac 1{2\gamma} \sum_{x \in \bb T_L}  
\frac 1 L \sum_{z\in \bb T_L} \Big\{ \sum_{z_i \in \Gamma (x,z)} 
H(\xi + \mf d_{z_{i+1}}) -  H(\xi + \mf d_{z_i}) \Big\}^2 \;,
\end{split}
\end{equation*}
where $\Gamma(x,z)$ stands for a path $(x=z_0, \dots, z_m=z)$ from $x$
to $z$ such that $|z_i-z_{i+1}|=1$ for $0\le i<m$. Since we may find
paths whose length are less than or equal to $L$, we may bound the
second sum using Schwarz inequality. After a change in the order of
summation this term becomes
\begin{equation*}
\frac 1{2\gamma} \sum_{w\in \bb T_L} 
\big\{ H(\xi + \mf d_{w}) -  H(\xi + \mf d_{w+1}) \big\}^2
\sum_{x,z\in \bb T_L}  \; ,
\end{equation*}
where the second sum is carried over all states $x$, $z\in \bb T_L$
whose path $\Gamma(x,z)$ passes through the bond $(w,w+1)$. This sum
is clearly less than or equal to
\begin{equation*}
\frac {L^2}{2\gamma}  \sum_{w\in \bb T_L} 
\big\{ H(\xi + \mf d_{w+1}) -  H(\xi + \mf d_{w}) \big\}^2 \; .
\end{equation*}

Up to this point we proved that $\<\mc L F , H\>_{\mu_N}$ is absolutely
bounded by 
\begin{equation*}
\begin{split}
& \frac \gamma 2 \frac{W_{N-1}}{W_{N}} \sum_{\xi\in E_{N-1}}
\mu_{N-1}(\xi) \sum_{x \in \bb T_L} 
[F(\xi + \mf d_{x+1}) - F(\xi + \mf d_x)]^2\, \\
& \quad +\; \frac {L^2}{2\gamma} 
\frac{W_{N-1}}{W_{N}} \sum_{\xi\in E_{N-1}} \mu_{N-1}(\xi)
\sum_{x \in \bb T_L} 
\big\{H(\xi + \mf d_{x+1}) -  H(\xi + \mf d_{x}) \big\}^2 \;.
\end{split}
\end{equation*}
After a change of variables, we bound this expression by
\begin{equation*}
\gamma \, \<(-\mc L) F , F\>_{\mu_N} \;+\; \frac {L^2}{\gamma} 
\<(-\mc L) H , H\>_{\mu_N}\;.
\end{equation*}
To conclude the proof it remains to optimize over $\gamma$.
\end{proof}

Denote by $\Cap^s_N(\ms A,\ms B)$ the capacity between two disjoint
subsets $\ms A$, $\ms B$ of $E_N$ with respect to the
\emph{reversible} zero range process with generator $\mc S$ given by
\eqref{11}:
\begin{equation*}
\Cap^s_N (\ms A,\ms B)\,=\, \inf_F D_N(F) \;,
\end{equation*}
where the infimum is carried over all functions $F$ which are equal to
$1$ at $\ms A$ and $0$ at $\ms B$. The next result follows from
\cite[Lemma 2.5 and 2.6]{gl2} and Lemma \ref{u02} above.

\begin{lemma}
\label{s02}
For every subsets $\ms A$, $\ms B$ of $E_N$, $\ms A\cap \ms B = \varnothing$,
\begin{equation*}
\Cap^s_N (\ms A,\ms B) \;\le\; \Cap_N (\ms A,\ms B) \;\le\; 4\, L^2\, 
\Cap^s_N (\ms A,\ms B)\;.
\end{equation*}
\end{lemma}

Denote by $\Cap^s(x,y)$, $x\not = y \in \bb T_L$, the capacity between
$x$ and $y$ for the nearest-neighbor symmetric random walk on $\bb
T_L$ which jumps with rate $1/2$ to the right and rate $1/2$ to the
left. For a proper subset $A$ of $\bb T_L$, let
\begin{equation}
\label{10}
\ms C_\alpha (A,A^c) \;=\; \frac 1{\Gamma(\alpha) \, I_\alpha} 
\sum_{x \in A, y\in A^c} \Cap^s(x,y)\;.
\end{equation}

\begin{lemma}
\label{s01}
Fix a proper subset $A$ of $\bb T_L$. Then,
\begin{equation*}
\begin{split}
& \ms C_\alpha (A,A^c) \; \le\; \liminf_{N\to\infty} N^{1+\alpha} 
\Cap_N\big(\ms E_N(A), \ms E_N(A^c)\big) \\
&\qquad \;\le\; \limsup_{N\to\infty} N^{1+\alpha} 
\Cap_N\big(\ms E_N(A), \ms E_N(A^c)\big) \; \le\; 
4 \, L^2 \, \ms C_\alpha (A,A^c) \;.
\end{split}
\end{equation*}
\end{lemma}

\begin{proof}
By \cite[Theorem 2.1]{bl3}, for every proper subset $A$ of $\bb T_L$,
\begin{equation}
\label{34}
\lim_{N\to\infty} N^{1+\alpha} \Cap^s_N\big(\ms E_N(A), \ms
E_N(A^c)\big) \;=\; \ms C_\alpha (A,A^c) \;.
\end{equation}
The result follows now from Lemma \ref{s02}.
\end{proof}

\section{Upper bound}
\label{sec4}

We prove in this section the upper bound for the capacity.  The
following decomposition of the generator $\mc L$ as the sum of cycle
generators will prove to be most helpful.  Fix a configuration $\xi\in
E_{N-1}$ and denote by $\mc L_\xi$ the generator on $E_N$ given by
\begin{equation*}
(\mc L_\xi f)(\eta) \;=\; \sum_{x\in \bb T_L} \mb 1\{\eta = \xi + \mf
d_x\} \, g(\eta(x)) \, [f(\sigma^{x,x+1} \eta) - f(\eta)]\;.
\end{equation*}
Note that the generator $\mc L$ restricted to $E_N$ may be written as
\begin{equation*}
\mc L \;=\; \sum_{\xi\in E_{N-1}} \mc L_\xi  
\end{equation*}
and that the measure $\mu_{N}$ is stationary for each generator
$\mc L_\xi$. Moreover, for any pair of functions $f$, $h:E_N\to \bb
R$, 
\begin{equation}
\label{01}
\< f, \mc L_\xi h \>_{\mu_{N}} \;=\; \frac{N^\alpha}{Z_N} \,\frac 1{a(\xi)}\, 
\sum_{x=1}^L f(\xi+\mf d_x) \, 
\{h(\xi+\mf d_{x+1}) - h(\xi+\mf d_x) \} \;.
\end{equation}
In particular, the adjoint of $\mc L_\xi$ in $L^2(\mu_N)$, denoted
by $\mc L^*_\xi$, is given by
\begin{equation*}
(\mc L^*_\xi f)(\eta) \;=\; \sum_{x\in \bb T_L} \mb 1\{\eta = \xi + \mf
d_x\} \, g(\eta(x)) \, [f(\sigma^{x,x-1} \eta) - f(\eta)]\;,
\end{equation*}
and the Dirichlet form $D_\xi$ associated to the generator $\mc L_\xi$
is given by
\begin{equation}
\label{03}
D_\xi (f)\;:=\; \< f, (-\mc L_\xi) f \>_{\mu_{N}} \;=\;
\frac{N^\alpha}{2 Z_N} \,\frac 1{a(\xi)}\, 
\sum_{x=1}^L \{f(\xi+\mf d_{x+1}) - f(\xi+\mf d_x) \}^2\;.
\end{equation}

\begin{proposition}
\label{s08}
Consider a sequence $\{\ell_N: N\ge 1\}$ satisfying \eqref{f18}. Fix a
proper subset $A$ of $\bb T_L$. Then,
\begin{equation*}
\limsup_{N\to\infty} N^{1+\alpha} \Cap_N\big(\ms E_N(A), \ms
E_N({A^c})\big) \;\le\; \frac 1{\Gamma(\alpha) \, I_\alpha} 
\sum_{x \in A, y\not\in A} \Cap(x,y) \; .
\end{equation*}
\end{proposition}

\begin{proof}
Fix a subset $A$ of $\bb T_L$.  For $N\ge 1$, $x\in \bb T_L$ and a
subset $C$ of $\bb T_L$, let
\begin{equation*}
\ms D^x_N\;:=\; \{\eta\in E_N : \eta_x \ge N - 3\ell_N \}\;, 
\quad \ms D_N (C) \;:=\; \bigcup_{x\in C} \ms D^x_N \;,
\end{equation*}
so that $\ms E^x_N \subset \ms D^x_N$, $ \ms E_N(C) \subset \ms D_N
(C)$. Therefore, by \cite[Lemma 2.2]{gl2}, $\Cap_N (\ms E_N (A)$, $\ms
E_N (A^c)) \le \Cap_N (\ms D_N (A), \ms D_N (A^c))$.  In particular,
to prove Proposition \ref{s08} it is enough to exhibit a function
$F_A$ in $\mc C_{1,0} (\ms D_N(A), \ms D_N(A^c))$ such that
\begin{equation}
\label{05}
\begin{split}
& \limsup_{N\to\infty} N^{1+\alpha} \sup_{h\in 
\mc C (\ms D_N(A), \ms D_N(A^c) )} \big\{
2\< F_A, \mc L h \>_{\mu_N} - D_N(h) \big\} \\
&\qquad \;\le\; \frac 1{\Gamma(\alpha) \, I_\alpha} 
\sum_{x \in A, y\not\in A} \Cap(x,y)\;. 
\end{split}
\end{equation}

The definition of the function $F_A$ requires some notation.  Fix an
arbitrary $0<\epsilon\ll 1$ and let $\bb W=\bb W_\epsilon:[0,1]\to
[0,1]$ be the smooth function given by
$$
\bb W(t)\;:=\; \frac{1}{I_{\alpha}} \,\int_0^{\phi(t)} 
u^{\alpha}(1-u)^{\alpha}\,du\;,
$$
where $I_{\alpha}$ is the constant defined in (\ref{defi}) and
$\phi:[0,1]\to [0,1]$ is a smooth non-decreasing function such that
$\phi(t)+\phi(1-t)=1$ for every $t\in [0,1]$ and $\phi(s)=0$ $\forall
s\in[0, 3\epsilon]$. It can be easily checked that
\begin{equation}\label{ph1}
\bb W(t)+\bb W(1-t)=1 \;,\quad \forall t\in [0,1]\;,
\end{equation}
and that $\bb W|_{[0,3\epsilon]} \equiv 0$ and $\bb W|_{[1-3\epsilon,1]}\equiv
1$. 

Let $\mb D\subset \bb R^{L}$ be the compact subset
$$
\mb D \;:=\; \{ u\in \bb R_{+}^{L} : \sum_{x\in \bb T_L} u_x = 1 \}\;.
$$
For each pair of sites $x\not =y\in \bb T_L$ and $\epsilon>0$, consider
the subsets of $\mb D$
\begin{equation}
\label{12}
\mb R^x_\epsilon \;:=\;\{ u\in \mb D : u_x \le \epsilon \} 
\quad \text{and} \quad 
\mb L^{xy}_{\epsilon} \;:=\; \{ u\in \mb D : u_x + u_y \ge 1-\epsilon
\}\;.   
\end{equation}
Clearly $\mb L^{xy}_{\epsilon} = \mb L^{yx}_{\epsilon}$ for any $x,y\in
\bb T_L$.

Fix $x$ in $\bb T_L$ and define $W_x : \mb D\to [0,1]$ as
follows. First define a function $\hat W_x$ on the set $\bigcup_{y\not
  = x} \mb L^{xy}_{\epsilon} \cup \mb R^x_\epsilon$ by
\begin{equation*}
\hat W_x (u) \;=\; 
\begin{cases}
(1/2)\, \big\{ \, \bb W(u_x) \,+\, [1-\bb W(u_y)] \, \big\} & \text{for $u\in
  \mb L^{xy}_{\epsilon}$}\;, y\not =x \;, \\
0 & \text{if $u\in \mb R^x_\epsilon$}\;.
\end{cases}
\end{equation*}
Note that $\hat W_x$ is well defined because $\hat W_x(u)=1$ for $u\in
\mb L^{xy}_{\epsilon} \cap \mb L^{xz}_{\epsilon}$, $y\not =z$, and
$\hat W_x(u)=0$ for $u\in \mb R^{x}_{\epsilon} \cap \mb
L^{xy}_{\epsilon}$, $y\not =x$. Let $W_x : \mb D\to [0,1]$ be a
Lipshitz continuous function which coincides with $\hat W_x$ on
$\bigcup_{y\not = x} \mb L^{xy}_{\epsilon} \cup \mb R^x_\epsilon$. 

Let $F_x:E_N\to \bb R$ be given by
\begin{equation}
\label{36}
F_x(\eta) \;:=\; W_x(\eta/N)\;,  
\end{equation}
where each $\eta/N$ is thought of as a point in $\mb D$. It follows
from the definition of $W_x$ that
\begin{equation}
\label{pf0}
\begin{split}
& F_x(\eta) \;=\; (1/2)\, \big\{ \, \bb W(\eta_x/N) 
\,+\, [1-\bb W(\eta_y/N)] \, \big\} \quad \textrm{for $\;\eta/N \in 
\mb L^{xy}_{\epsilon}$}\;, \\
& \quad F_x\equiv 1\quad \text{on $\{\eta\in E_N : \eta_x\ge
  (1-\epsilon)N\}$} \\
&\qquad F_x\equiv 0\quad \text{on $\{\eta\in E_N : \eta_x\le
  \epsilon N\}$} \;.
\end{split}
\end{equation}
Moreover, since $W_x$ is Lipschitz continuous, there exists a finite
constant $C_\epsilon$, which depends only on $\epsilon$, such that
\begin{equation}
\label{pf1}
\max_{z\in\bb T_L} \max_{\eta\in E_N} | F_x(\sigma^{z, z+1}\eta) -
F_x(\eta) | \;\le\; \frac{C_{\epsilon}}{N}\;\cdot
\end{equation}

Recall that we fixed a nonempty subset $A\subsetneq \bb T_L$. Define
the function $F_{A}:E_N\to \bb R$ as
\begin{equation*}
F_{A}(\eta) \;:=\; \sum_{x\in A} F_x(\eta) \;.  
\end{equation*}
The function $F_A$ is our candidate to estimate the left hand side of
\eqref{05}. 

It follows from (\ref{pf0}) that if $\eta\in \ms D^x_N$ for some $x\in
\bb T_L$ then for $N$ large enough
\begin{equation*}
F_{A}(\eta) \;=\; {\bf 1}\{x\in A\}\;=\; F_{A}(\sigma^{zw}\eta)\;,
\end{equation*}
for every $z,w\in \bb T_L$ and every $N$ large enough. In particular,
$$
F_{A} \;\in\; \mc C_{1,0} \big( \, \ms D_N(A) , \ms D_N(A^c) \, \big) \;.
$$

It remains to prove \eqref{05}. For $N\ge 1$ and $x$, $y\in \bb T_L$,
$x\not=y$, let
$$
\ms I^{xy}_N \;:=\; \big\{ \eta\in E_N : \eta_x 
+ \eta_y \ge N - \ell_N \big\} \;.
$$
Clearly, $\ms I^{xy}_N=\ms I^{yx}_N$, $x\not = y\in \bb T_L$, and, for
every $N$ large enough, $\eta/N$ belongs to $\mb L^{xy}_\epsilon$ if
$\eta$ belongs to $\ms I^{xy}_N$. Moreover, for $N$ sufficiently
large,
\begin{equation}
\label{02}
\begin{split}
& \ms I^{x,y}_{N} \cap \ms I_{N}^{z,w}\not = \varnothing \quad 
\textrm{if and only if} \quad\{x,y\}\cap\{z,w\}\not =\varnothing \;, \\    
& \quad 
\ms I^{x,y}_{N} \cap \ms I^{x,z}_{N} \;\subset\; \{\eta \in E_{N}
: \eta_x \ge N-2\ell_N\}\;, \quad y,z \not = x \;. 
\end{split}
\end{equation}

Let $\ms R_{N} = E_N \setminus \{\, \bigcup_{x\not = y} \ms
I^{x,y}_{N}\, \}$, and let $\mc L_{R}$, $\mc L_{x,y}$, $x\not = y\in
\bb T_L$, be the generators on $E_N$ given by
\begin{equation*}
\mc L_{x,y} \;=\; \sum_{\xi \in \ms I^{x,y}_{N-1}} \mc L_\xi \;, \quad
\mc L_{R} \;=\; \sum_{\xi \in \ms R_{N-1}} \mc L_\xi \;.
\end{equation*}
Note that $N$ has been replaced by $N-1$ so that each configuration
$\xi$ in this formula belongs to $E_{N-1}$.  Even though the
generators $\mc L_{x,y}$ and $\mc L_{x,z}$ have common factors $\mc
L_\xi$, in view of \eqref{02}, for a function $f$ constant on each set
$\ms D^w_N$, $w\in \bb T_L$,
\begin{equation*}
\mc L f \;=\; \sum_{y\not = z} \mc L_{y,z} f \;+\; \mc L_R f\;.
\end{equation*}
The first sum is carried over all pairs of sites $\{y,z\}$, each pair
appearing only once.  In particular, for functions $f$, $h$ in $\mc
C(\ms D_N(A) , \ms D_N(A^c))$,
\begin{equation*}
\<f, \mc L h \>_{\mu_N} \;=\;
\sum_{y\not = z} \<f, \mc L_{y,z} h \>_{\mu_N} \;+\;
\<f, \mc L_{R} h \>_{\mu_N}\;.
\end{equation*}
Therefore, 
\begin{equation}
\label{08}
\begin{split}
& \sup_{h \in \mc C (\ms D_N(A), \ms D_N(A^c))} 
\big\{ 2\<F_A, \mc L h \>_{\mu_N} \;-\; \<h, (- \mc L) h \>_{\mu_N} \big\} \\
& \qquad\qquad \;\le\;
\Big( \sum_{\substack{y, z \in A \\ y\not = z}} + 
\sum_{\substack{y, z \in A^c \\ y\not = z}} \Big)
\sup_{h} \big\{ 2\<F_A, \mc L_{y,z} h \>_{\mu_N} 
\;-\; \<h, (-\mc L_{y,z}) h \>_{\mu_N} \big\} \\
& \qquad\qquad \;+\;
\sum_{y\in A, z\not\in A} \sup_{h\in \mc C(\ms D^y_N, \ms D^z_N)} 
\big\{ 2\<F_A, \mc L_{y,z} h \>_{\mu_N} 
\;-\; \<h, (-\mc L_{y,z}) h \>_{\mu_N} \big\} \\
&\qquad\qquad \;+\; \sup_{h} 
\big\{ 2\<F_A, \mc L_{R} h \>_{\mu_N} 
\;-\; \<h, (-\mc L_{R}) h \>_{\mu_N} \big\}\;.
\end{split}
\end{equation}
In view of \eqref{05}, to complete the proof of Proposition \ref{s08},
it remains to show that the limsup of the right hand side multiplied
by $N^{1+\alpha}$ is bounded by the right hand side of \eqref{05}.  We
estimate separately each piece of this decomposition.

We start with the first term on the right hand side.  If $\eta/N$
belongs to $\mb L^{xy}_{\epsilon}$ for some $x$, $y$ in $A$, $x \not =
y$, by \eqref{pf0}, $F_{A}(\eta) = F_{x}(\eta) + F_{y}(\eta) = 1$.
Similarly, if $\eta/N$ belongs to $\mb L^{xy}_{\epsilon}$ for some
$x$, $y$ in $A^c$, $x \not = y$, $F_{A}(\eta) = 0$.  Hence, for any $N$
large enough,
\begin{equation*}
\begin{split}
& F_{A}(\sigma^{zw}\eta) \;=\; F_{A}(\eta) \;= \; 1   
\quad \textrm{for all } \eta \in \bigcup_{x, y\in A} 
\ms I^{xy}_N \text{ and } z,w\in \bb T_L \;, \\
&\quad F_{A}(\sigma^{zw}\eta) \;=\; F_{A}(\eta) \;=\; 0 
\quad \textrm{for all } \eta \in \bigcup_{x, y\not\in A} 
\ms I^{xy}_N \text{ and } z,w\in \bb T_L\;.
\end{split}
\end{equation*}
Therefore, if $y$, $z\in A$, or if $y$, $z\in A^c$, $\mc L^*_{y,z} F_A
=0$, where $\mc L^*_{y,z}$ is the adjoint of $\mc L_{y,z}$ in
$L^2(\mu_N)$, so that
\begin{equation*}
\sup_{h} \big\{ 2\<F_A, \mc L_{y,z} h \>_{\mu_N} 
\;-\; \<h, (-\mc L_{y,z}) h \>_{\mu_N} \big\}\;=\;0\;.
\end{equation*}

Consider now the second term on the right hand side of \eqref{08}. Fix
$y\in A$, $z\not \in A$. We claim that
\begin{equation}
\label{04}
\begin{split}
& \sup_{h\in \mc C(\ms D^y_N, \ms D^z_N)} 
\big\{ 2\<F_A, \mc L_{y,z} h \>_{\mu_N} 
\;-\; \<h, (-\mc L_{y,z}) h \>_{\mu_N} \big\}  \\
& \qquad\qquad \le\;
\frac{N^\alpha}{Z_N} 
\sum_{\xi\in \ms I^{y,z}_{N-1}} \,\frac 1{a(\xi)}\, 
\big[ \bb W ([\xi_y+1]/N) - \bb W (\xi_y/N)\big]^2 \;.
\end{split}
\end{equation}

Indeed, in view of \eqref{01} we have that
\begin{equation*}
2\<\mc L^*_{y,z} F_A,  h \>_{\mu_N} \;=\; \frac{ 2 N^\alpha}{Z_N} 
\sum_{\xi\in \ms I^{y,z}_{N-1}} \,\frac 1{a(\xi)}\, 
\sum_{x=1}^L h (\xi+\mf d_x) \, 
\{F_A(\xi+\mf d_{x-1}) - F_A (\xi+\mf d_x) \} \;.
\end{equation*}
Since for any configuration $\eta$ which can be written as $\xi+\mf
d_w$ for some $\xi \in \ms I^{y,z}_{N-1}$, $w\in\bb T_L$, $F_A (\eta)
= F_y(\eta) = (1/2) \{ \bb W(\eta_y/N) + [1-\bb W(\eta_z/N)]\}$ the sum over
$x$ becomes
\begin{equation*}
\begin{split}
& (1/2) \{h (\xi+\mf d_{y+1}) - h (\xi+\mf d_{y})\}
\, \{\bb W ([\xi_y+1]/N) - \bb W (\xi_y/N)\} \\
&\qquad -\;
(1/2) \{h (\xi+\mf d_{z+1}) - h (\xi+\mf d_{z})\}
\, \{\bb W ([\xi_z+1]/N) - \bb W (\xi_z/N)\}\;.   
\end{split}
\end{equation*}
Hence, by Schwarz inequality, $2\<\mc L^*_{y,z} F_A,  h \>_{\mu_N}$ is
absolutely bounded by
\begin{equation*}
\begin{split}
& \frac{N^\alpha}{2 Z_N} 
\sum_{\xi\in \ms I^{y,z}_{N-1}} \,\frac 1{a(\xi)}\, \Big\{
\big[ \bb W ([\xi_y+1]/N) - \bb W (\xi_y/N)\big]^2 + 
\big [ \bb W ([\xi_z+1]/N) - \bb W (\xi_z)\big]^2 \Big\} \\
&\quad +\; \frac{N^\alpha}{2 Z_N} 
\sum_{\xi\in \ms I^{y,z}_{N-1}} \,\frac 1{a(\xi)}\, \Big\{
\big[ h (\xi+\mf d_{y+1}) - h (\xi+\mf d_{y}) \big]^2 + 
\big [ h (\xi+\mf d_{z+1}) - h (\xi+\mf d_{z}) \big]^2 \Big\}\;.
\end{split}
\end{equation*}
By \eqref{03}, the second term is bounded above by $\sum_{\xi\in
  \ms I^{y,z}_{N-1}} \< h, (-\mc L_\xi) h \>_{\mu_{N}} = \<h, (-\mc L_{y,z}) h
\>_{\mu_N}$. On the other hand, since the set $\ms I^{y,z}_{N-1}$ is symmetric
in $y$ and $z$, the first line coincides with the right hand side of
\eqref{04}, which concludes the proof of this claim.

It remains to examine the last term of \eqref{08}. We claim that
\begin{equation}
\label{06}
\lim_{N\to\infty} \sup_{h} \big\{ 2\<F_A, \mc L_{R} h \>_{\mu_N} 
\;-\; \<h, (-\mc L_{R}) h \>_{\mu_N} \big\} \;=\; 0\;. 
\end{equation}
Indeed, by the strong sector condition the supremum on the left hand
side of this identity is bounded by $C_0 \<F_A, (-\mc L_{R}) F_A
\>_{\mu_N}$ for some finite constant $C_0$ depending only on $L$. By
definition of $\mc L_R$, this expression is equal to
\begin{equation*}
C_0  \sum_{\xi \in \ms R_{N-1}} \<F_A, (-\mc L_\xi) F_A \>_{\mu_N} 
\;\le\; \frac{C_\epsilon}{N^{\alpha+1} \, \ell_N^{\alpha -1}} \;,
\end{equation*}
where the last estimate follows from Lemma \ref{est1} below.

Up to this point we proved that the left hand side of \eqref{05} is
bounded above by
\begin{equation*}
\limsup_{N\to\infty} \frac{N^{1+2\alpha}}{Z_N} \sum_{y\in A, z\not\in A}
\sum_{\xi\in \ms I^{y,z}_{N-1}} \,\frac 1{a(\xi)}\, 
\big[ \bb W ([\xi_y+1]/N) - \bb W (\xi_y/N)\big]^2 \;.
\end{equation*}
Proposition 2.1 in \cite{bl3}, the explicit expression of
$\bb W_\epsilon$, and a simple computation permits to show that this
expression converges, as $\epsilon\downarrow 0$, to
\begin{equation*}
\frac 1{\Gamma(\alpha) \, I_\alpha} \frac{|A|\, (L-|A|)}L\;=\;
\frac 1{\Gamma(\alpha) \, I_\alpha} 
\sum_{x \in A, y\not\in A} \Cap(x,y)\;.
\end{equation*}
This concludes the proof of Proposition \ref{s08}. 
\end{proof}

We close this section with an estimate used above.

\begin{lemma}
\label{est1}
For every $x\in \bb T_L$ and every $N$ large enough,
$$
\frac{N^\alpha}{Z_N} \sum_{\xi \in \ms R_{N-1}} \frac 1{a(\xi)}\, 
\sum_{z=1}^L \{F_x (\xi+\mf d_{z+1}) - F_x (\xi+\mf d_z) \}^2
\;\le \;  \frac{C_\epsilon}{N^{\alpha+1} 
\, \ell_N^{\alpha -1}}\;\cdot
$$
\end{lemma}

The proof of this lemma is similar to the one of Lemma 5.2 in
\cite{bl3} and therefore omitted.

\section{Lower bound}
\label{sec3}

\begin{proposition}
\label{s05} 
Suppose that $\alpha>3$.  Let $\{\ell_N : N\ge 1\}$ be a sequence
satisfying \eqref{f18} and let $A$ be a proper subset of $\bb
T_L$. Then,
\begin{equation*}
\liminf_{N\to\infty} N^{1+\alpha} \Cap_N\big(\ms E_N(A), \ms
E_N({A^c})\big) \;\ge\; \frac 1{\Gamma(\alpha) \, I_\alpha} 
\sum_{x \in A, y\not\in A} \Cap(x,y) \; .
\end{equation*}
\end{proposition}

The proof of this statement relies on three observations. On the one
hand, on any strip $\{\eta \in E_N : \eta/N \in \mb
L^{xy}_\epsilon\}$, $x\not = y$, where $\mb L^{xy}_\epsilon$ has been
defined in \eqref{12}, $\eta_{x+1}+ \cdots + \eta_y$ behaves
essentially as a birth and death process with birth rate $g(\eta_x)$
and death rate $g(\eta_y)\approx g(N-\eta_x)$. This means that on each
strip $\mb L^{xy}_\epsilon$ the variable $\eta_{x+1}+ \cdots + \eta_y$
evolves as a symmetric zero-range process on two sites whose
metastable behavior is easy to determine \cite{bl7}.

Secondly, as we said just after the statement of Theorem \ref{mt2},
the variational formula \eqref{09} for the capacity has to be
understood as an infimum over a class of functions $H$ satisfying
certain relations. The main object in this formula is $H$ and not $F$
as one may think.

Finally, we shall use in the argument the monotonicity of the capacity
stated in \cite[Lemma 2.2]{gl2}: if $\ms A\subset \ms A'$ and $\ms
B\subset \ms B'$,
\begin{equation}
\label{26}
\Cap (\ms A, \ms B) \;\le\; \Cap (\ms A', \ms B')\;. 
\end{equation}

These three observations lead to the following approach for the proof
of the lower bound. For a fixed function $f$, we first consider the
variational problem \eqref{09} with the generators $\mc L$ and $\mc S$
restricted to tubes contained in the strips $\mb L^{xy}_\epsilon$. For
this problem we optimize over functions $h=h(\eta_{x+1}+ \cdots +
\eta_y)$ to obtain a lower bound in terms of the Dirichlet form of $h$
with respect to a symmetric zero-range dynamics over two sites. We
then estimate the remaining piece of the original variational problem
by extending the function $h$ defined on the union of strips to the
entire space and by bounding its Dirichlet form on this space. By
\eqref{26}, we are allowed during this procedure to reduce the sets
$\ms E_N(A)$ and $\ms E_N({A^c})$ whenever necessary. 

This extension procedure is not difficult if $\ell_N \gg
N^{(1+\alpha)/(2\alpha -2)}$ and is more demanding when this bound
does not hold.  We recommend to the reader to assume below that $L=3$
in which case the dynamics can be viewed as a random walk on the
simplex $\{(i,j) : i\ge 0, j\ge 0, i+j\le N\}$.

\begin{proof}[Proof of Proposition \ref{s05}]
Fix a subset $A$ of $\bb T_L$ and a sequence of functions $f_N$ in
$\mc C_{1,0}(\ms E_N(A), \ms E_N(A^c))$. We have to show that
\begin{equation}
\label{23}
\liminf_{N\to\infty} \, N^{1+\alpha} \sup_{h} 
\Big\{ 2 \<   f_N \,,\, \mc L h\>_{\mu_N}  \, - 
\, \< h , (- \mc S) h\>_{\mu_N} \Big\}
\;\ge\; \frac 1{\Gamma(\alpha) \, I_\alpha} 
\,\frac {|A| \, (L-|A|)}{L} \; ,
\end{equation}
where the supremum is carried over functions $h$ in $\mc C(\ms E_N(A),
\ms E_N(A^c))$.

Since we know from \cite[Theorem 2.4]{gl2} that the function $F$ which
solves the variational problem \eqref{09} is $F= (1/2) \{ V_{\ms A,
  \ms B} + V^*_{\ms A, \ms B} \}$, we may restrict our attention to
functions $f_N$ which possess certain properties. We may assume, for
instance, that $f_N$ is non-negative and bounded above by $1$.  Lemma
\ref{s01} permits also to assume that
\begin{equation*}
N^{1+\alpha} D_N(f_N) \;\le\; 8 L^2 \ms C_\alpha(A,A^c)
\end{equation*}
for $N$ large enough, where $\ms C_\alpha(A,B)$ has been introduced in
\eqref{10}. Indeed, by Lemma \ref{s01}, $N^{1+\alpha} \Cap_N\big(\ms
E_N(A), \ms E_N(B)\big) \le 8 L^2 \ms C_\alpha(A,B)$ for $N$
sufficiently large. On the other hand, taking $h=-f_N$ we get that
\begin{equation*}
\sup_h \Big\{ 2 \<  f_N \,,\, \mc Lh \>_{\mu_N}  \, - 
\, \< h , (- \mc S) h\>_{\mu_N} \Big\} \;\ge\; D_N(f_N)\;,
\end{equation*}
which proves the claim.

To fix ideas, consider the case where $A$ is the singleton
$\{0\}$. Note, incidentally, that the proof of the metastability of a
sequence of Markov processes presented in \cite{bl2} requires only
Theorem \ref{mt1} for sets $A$ which are either singletons or
pairs. It will be clear from the proof, however, that the general case
where $A$ is any subset of $\bb T_L$ can be treated similarly.

The proof is divided in three steps. We first examine the expression
inside braces in \eqref{23} along tiny strips in the space of
configurations. Fix a non-negative function $f$ in $\mc C_{1,0}(\ms
E^0_N, \ms E_N(\{0\}^c))$ bounded by $1$ and such that
\begin{equation}
\label{19}
N^{1+\alpha} D_N(f) \;\le\; 8 L^2 \ms
C_\alpha(\{0\} , \{0\}^c)\; .
\end{equation}

\smallskip\noindent{\bf Step 1: The main contributions.}  Fix a
sequence $\{k_N : N\ge 1\}$, $1\ll k_N \ll \ell_N$.  For $x\not = y$,
consider the strips
\begin{equation*}
\ms J_{x,y} = \{\zeta \in E_{N-1}: \zeta_z \le k_N \,,\, z\not = x,
y\}\;, \quad \ms J^+_{x,y} = \{\zeta + \mf d_z : \zeta \in \ms J_{x,y} , 
z\in \bb T_L\}\;,
\end{equation*}
and let
\begin{equation*}
\ms J_x = \bigcup_{y\not = x} \ms J_{x,y}\;, \quad 
\ms J_x^+ = \bigcup_{y\not = x} \ms J^+_{x,y}\;.
\end{equation*}
We claim that there exists a function $h: \ms J_0^+ \to [0,1]$,
$h(\eta) = H(\eta_{1}+ \cdots + \eta_x)$ on $\ms J^+_{0,x}$, such that
\begin{equation}
\label{25}
N^{1+\alpha} \sum_{\zeta\in\ms J_0} \Big\{ 2\, \< f, \mc L_\zeta h\>_{\mu_N}
\;-\; \< h, (-\mc L_\zeta) h\>_{\mu_N} \Big\} \;\ge\; 
\frac {L-1}{L\, \Gamma(\alpha) \, I_\alpha} \;-\; o_N(1)\;,
\end{equation}
where $o_N(1)$ represents a constant which vanishes as
$N\uparrow\infty$.

Fix $x\not =0$ and consider a function $h(\eta) = H(\eta_{1}+ \cdots +
\eta_x)$ on $\ms J^+_{0,x}$.  We first compute
\begin{equation}
\label{14}
\sum_{\zeta\in\ms J_{0,x}} \< f, \mc L_\zeta h\>_{\mu_N}\;.
\end{equation}
In view of the definition of the generator $\mc L_\zeta$ and the
property of $h$, this expression is equal to
\begin{equation*}
\frac{N^\alpha}{Z_N} \sum_{\zeta\in \ms J_{0,x}} \frac 1{a(\zeta)} 
\big\{f(\zeta + \mf d_0) -  f(\zeta + \mf d_x) \big\} \, 
\big\{H(\zeta_{(0x]} +1) - H(\zeta_{(0x]}) \big\}  \;,
\end{equation*}
where $\zeta_{(0x]} = \zeta_{1}+ \cdots + \zeta_x$. Decompose this sum
according to the possible values of $\zeta_{(0x]}$ to rewrite it as
\begin{equation*}
\frac{N^\alpha}{Z_N} \sum_{i=0}^{N-1} [H(i+1) - H(i)] \,
\sum_{\substack{\zeta\in \ms J_{0,x}\\ \zeta_{(0x]} =i}} \frac 1{a(\zeta)} 
\big\{f(\zeta + \mf d_0) -  f(\zeta + \mf d_x) \big\}\;.
\end{equation*}
Since $k_N\ll \ell_N$ and since $f$ is equal to $1$ on $\ms
E_N(\{0\})$, all terms in the second sum vanish if $i\le L k_N$.  In
particular, since $\zeta_{(0x]} =i$, the second sum is carried over
all configurations $\zeta$ such that $\zeta_y \le k_N$, $y= 1, \dots,
x-1, x+1, \dots, L-1$, and $\zeta_x = i - \sum_{1\le y <x}
\zeta_y$. This observation will be used several times below.

Let $\bar N= N-1$ and recall that $\sum_{y\in\bb T_L} \zeta_y =\bar
N$. Denote by $\xi$ the configuration $\zeta$ without the coordinates
$\zeta_0$ and $\zeta_x$: $\xi = (\zeta_1, \dots, \zeta_{x-1},
\zeta_{x+1}, \dots, \zeta_{L-1})$ and let $M_1 = \sum_{1\le y\le x-1}
\zeta_y$, $M_2 = \sum_{x+1\le y\le L-1} \zeta_y$.  With this notation,
we may rewrite the second sum appearing in the previous displayed
formula as
\begin{equation*}
\sum_{\xi \in \ms R_N} \frac {f(\bar N -i -M_2+1, i-M_1, \xi)
- f(\bar N -i -M_2, i+1-M_1, \xi)} {a(\bar N-i-M_2) a(i-M_1) a(\xi)}\;,  
\end{equation*}
where $\ms R_N =\{\xi\in\bb N^{\bb T_{L-2}} : \xi_y \le k_N\}$, and
where we changed the order of the coordinates of $\zeta$ starting with
the coordinates $\zeta_0$ and $\zeta_x$.

Let $F: \{0, \dots , N\} \to \bb R_+$ be given by
\begin{equation}
\label{16}
F(i) \;=\; \sum_{\xi \in \ms R_N} \frac  1 {a(\xi)}
f(\bar N -i -M_2+1, i-M_1, \xi) \;,
\end{equation}
so that for $N$ sufficiently large \eqref{14} is equal to
\begin{equation}
\label{15}
\frac{N^\alpha}{Z_N} \sum_{i=L k_N}^{N-1} 
\frac {1} {a(\bar N-i) a(i)} [H(i+1) - H(i)] \, [F(i) - F(i+1)] \; +\;
R_N\;, 
\end{equation}
where the remainder $R_N$ is given by 
\begin{equation*}
\begin{split}
& R_N \;=\; \frac{N^\alpha}{Z_N} \sum_{i=Lk_N}^{N-1} [H(i+1) - H(i)]
\sum_{\xi \in \ms R_N} \frac {(\nabla_{0,x}f)
(\bar N -i -M_2, i-M_1, \xi)}{a(\xi)} \;\times \\
& \qquad\qquad \qquad\qquad
\times \; \Big\{ \frac 1{a(\bar N-i-M_2) a(i-M_1) } -
\frac 1{a(\bar N-i) a(i)} \Big\}\;,
\end{split}
\end{equation*}
where $(\nabla_{0,x}f) (\bar N -i -M_2, i-M_1, \xi) =
f (\bar N -i -M_2+1, i-M_1, \xi) - f (\bar N -i -M_2, i+1-M_1, \xi)$.

We are now in a position to define $h$ on an $\ms J^+_{0,x}$. In view of
\eqref{16}, \eqref{15}, for $\eta\in \ms J^+_{0,x}$ set
\begin{equation}
\label{17}
h(\eta) \;=\; - \Big( \sum_{\xi \in \ms R_N} \frac  1 {a(\xi)} \Big)^{-1} 
\sum_{\xi \in \ms R_N} \frac  1 {a(\xi)}
f(N - \eta_{(0x]} - M_2, \eta_{(0x]} - M_1, \xi) \;,
\end{equation}
where we inverted again the order of the coordinates starting with
$\eta_0$, $\eta_x$ and where $M_1 = M_1(\xi) = \sum_{1\le y< x}\xi_y$,
$M_2= M_2(\xi) = \sum_{x< y< L}\xi_y$. Defined in this way $h$ is
clearly a function of $\eta_{(0,x]}$, $h(\eta)= H(\eta_{(0,x]})$,
$H(i)$ is equal to $1$ for $i\le \ell_N/2$, and equal to $0$ for $i\ge
N- (\ell_N/2)$, and
\begin{equation*}
\Big( \sum_{\xi \in \ms R_N} \frac  1 {a(\xi)} \Big) H(i) \;=\; -  
\sum_{\xi \in \ms R_N} \frac  1 {a(\xi)}
f(\bar N - i - M_2 +1, i - M_1, \xi) \;=\; - F(i)\; . 
\end{equation*}
Therefore, the first term of \eqref{15} becomes
\begin{equation*}
\frac{N^\alpha}{Z_N} \Big( \sum_{j=0}^{k_N} \frac  1 {a(j)}
\Big)^{(L-2)} \sum_{i=L k_N}^{N-1} 
\frac {1} {a(\bar N-i) a(i)} [H(i+1) - H(i)]^2 \;. 
\end{equation*}

A similar computation shows that for $h$ given by \eqref{17},
\begin{equation*}
\begin{split}
& \sum_{\zeta\in\ms J_{0,x}} \< h, (-\mc L_\zeta) h\>_{\mu_N} \;=\; \\
&\quad
\frac{N^\alpha}{Z_N}  \sum_{i=L k_N}^{N-1} 
[H(i+1) - H(i)]^2
\sum_{\xi \in \ms R_N} \frac  1 {a(\xi) a(i-M_1) a(\bar N - i - M_2)}
\; \cdot
\end{split}
\end{equation*}
Since $H$ is constant at distance $\ell_N/2$ from $0$ and from $N$,
and since $k_N\ll \ell_N$, $a(i-M_1)\ge a(i) (1-o_N(1))$, $a(\bar N -
i - M_2)\ge a(\bar N - i) (1-o_N(1))$. Therefore, the previous
expression is bounded above by
\begin{equation*}
(1+o_N(1)) \, \Big( \sum_{j=0}^{k_N} \frac  1 {a(j)}
\Big)^{(L-2)} \frac{N^\alpha}{Z_N}  \sum_{i=L k_N}^{N-1} 
\frac  1 {a(i) a(\bar N - i)} [H(i+1) - H(i)]^2 \; .
\end{equation*}

Up to this point we proved that defining $h$ by \eqref{17} on the
strip $\ms J^+_{0,x}$, we have that
\begin{equation*}
\begin{split}
& \sum_{\zeta\in\ms J_{0,x}}  \Big\{ 2 \< f, \mc L_\zeta h\>_{\mu_N} \;-\;
\< h, (-\mc L_\zeta) h\>_{\mu_N} \Big\} \;\ge\; \\
& \quad (1-o_N(1)) \, \Big( \sum_{j=0}^{k_N} \frac  1 {a(j)}
\Big)^{(L-2)} \frac{N^\alpha}{Z_N}  \sum_{i=L k_N}^{N-1} 
\frac  1 {a(i) a(\bar N - i)} [H(i+1) - H(i)]^2 \;+\; R_N\;.
\end{split}
\end{equation*}
By \eqref{zk}, $Z_N (\sum_{0\le j\le k_N} a(j)^{-1})^{-(L-2)}$
converges to $L \Gamma(\alpha)$. Due to the boundary conditions of $H$
at $0$ and $N$, $N^{1+2\alpha} \sum_{i} \{a(i) a(\bar N - i)\}^{-1}
[H(i+1) - H(i)]^2$ is bounded below by an expression which converges to
$I_{\alpha}^{-1}$ as $N\uparrow\infty$. We claim that $R_N$ vanishes
as $N\uparrow\infty$ and that 
\begin{equation}
\label{18}
N^{1+2\alpha}  \sum_{i=L k_N}^{N-1} 
\frac  1 {a(i) a(\bar N - i)} [H(i+1) - H(i)]^2 \;\le\; C_0  
\end{equation}
for some finite constant $C_0$ independent of $N$. Assuming that these
two claims are correct, defining $h$ by \eqref{17} on the strip $\ms
J^+_{0,x}$, we have that
\begin{equation*}
N^{1+2\alpha} \sum_{\zeta\in\ms J_{0,x}}  
\Big\{ 2 \< f, \mc L_\zeta h\>_{\mu_N} \;-\;
\< h, (-\mc L_\zeta) h\>_{\mu_N} \Big\} \;\ge\; 
\frac 1{L\, \Gamma(\alpha) \, I_\alpha} \;-\; o_N(1)\;.
\end{equation*}

The proof that $R_N$ vanishes is similar to the one of \eqref{18} and
relies on the fact that $k_N\ll \ell_N$. We present the arguments
which lead to \eqref{18} and leave to the reader the details of the
other assertion. By definition of $H$ and by Schwarz inequality, the
left hand side of \eqref{18} is bounded by
\begin{equation*}
\sum_{i=L k_N}^{N-1} 
\frac  {N^{1+2\alpha}} {a(i) a(\bar N - i)}  
\sum_{\xi \in \ms R_N} \frac  1 {a(\xi)}
\big\{ f(N - i - 1 - M_2, i+1 - M_1, \xi) -
f(N - i - M_2, i - M_1, \xi)\big\} ^2
\end{equation*}
because $(\sum_{\xi \in \ms R_N} a(\xi)^{-1})^{-1} \le 1$. The change
of variables $i=\zeta_{(0x]}$ permits to rewrite this sum as
\begin{equation*}
\sum_{\zeta\in \ms J_{0,x}}
\frac  {N^{1+2\alpha}} {a(\zeta_{(0x]}) \, a(\bar N - \zeta_{(0x]})}  
\prod_{y\not = 0,x} \frac  1 {a(\zeta_y)}
\big\{ f(\zeta + \mf d_x) - f(\zeta + \mf d_0)\big\} ^2\;.
\end{equation*}
Since $a(\zeta_{(0x]}) \ge a(\zeta_x)$ and $a(\bar N - \zeta_{(0x]})
\ge a(\bar N - \sum_{y\not = 0} \zeta_y) = a(\zeta_0)$, the previous
expression is bounded by
\begin{equation*}
\sum_{\zeta\in \ms J_{0,x}}
\frac  {N^{1+2\alpha}} {a(\zeta)}  
\big\{ f(\zeta + \mf d_x) - f(\zeta + \mf d_0)\big\} ^2\;.
\end{equation*}
By Schwarz inequality this expression is less than or equal to $2 L
N^{1+\alpha} Z_N \sum_{\zeta\in\ms J_{0,x}} \< f$, $\mc L_\zeta f
\>_{\mu_N}$ which is bounded by $C_0$ in view of \eqref{19}. This
proves \eqref{18} and assertion \eqref{25}.

Unfortunately, estimate \eqref{25} does not settle the question on the
strips $\ms J^+_0$. The alert reader certainly noticed that $h$ does
not belong to $\mc C(\ms E^0_N, \ms E_N(\{0\}^c))$ and some surgery is
necessary. The simplest way to overcome this difficulty seems to be to
modify the sets $\ms E^0_N$, $\ms E_N(\{0\}^c)$ and to use \eqref{26}.

Two configurations $\eta$, $\xi$ belonging to the strip $\ms
J^+_{0,x}$ are said to be equivalent, $\eta\equiv\xi$, if $\eta_{(0x]}
= \xi_{(0x]}$. For $y\in \bb T_L$, let 
\begin{equation*}
\widehat{\ms E}^y_N \;=\; 
\Big\{\eta\in \ms E^y_N : \xi \in \ms E^y_N 
\text{ for all } \xi \equiv\eta \Big\}\;.
\end{equation*}
Clearly, $\widehat{\ms E}^y_N \subset {\ms E}^y_N$ and both sets
differ at the boundary by a set with at most $O(k_N^{L-1})$ elements.

Let $\widehat{\ms E}_N = \widehat{\ms E}^0_N$, $\breve{\ms E}_N =
\cup_{x\not = 0} \widehat{\ms E}^x_N$. By \eqref{26}, $\Cap(\ms
E^0_N, \ms E_N(\{0\}^c)) \ge \Cap(\widehat{\ms E}_N$,
$\breve{\ms E}_N)$. Moreover, if $f$ is a function in $\mc
C_{1,0}(\widehat{\ms E}_N, \breve{\ms E}_N)$ the above
construction produces a function $h: \ms J^+_0 \to [0,1]$ which also
belongs to $\mc C_{1,0}(\widehat{\ms E}_N$, $\breve{\ms
  E}_N)$ and such that $h(\eta) = H(\eta_{(0x]})$ on $\ms
J^+_{0,x}$ and
\begin{equation}
\label{27}
N^{1+\alpha} \sum_{\zeta\in\ms J_0} \Big\{ 2\, \< f, \mc L_\zeta h\>_{\mu_N}
\;-\; \< h, (-\mc L_\zeta) h\>_{\mu_N} \Big\} \;\ge\; 
\frac {L-1}{L\, \Gamma(\alpha) \, I_\alpha} \;-\; o_N(1)\;.
\end{equation}

On the strips $\ms J^+_{x,y}$, $x, y \not = 0$, $x\not = y$, we set
$h$ to be identically equal to $0$. With this definition, $h$ still
belongs to $\mc C_{1,0}(\widehat{\ms E}_N$, $\breve{\ms E}_N)$ and the
previous inequality remains in force with the set $\ms J_0$ replaced
by $\ms J = \cup_{x\in\bb T_L} \ms J_x$. Note that configurations
which do not belong to this latter set have at least three sites with
at least $k_N$ particles in each.

\smallskip
\noindent{\bf Step 2. Extending the function $h$.} 
Denote by $\mc L_{B}$, $\mc L_{\complement}$ the generators $\mc L_B =
\sum_{\zeta\in \ms J} \mc L_\zeta$, $\mc L_{\complement} =
\sum_{\zeta\not\in \ms J} \mc L_\zeta$ so that $\mc L = \mc L_B +
\mc L_{\complement}$.

In view of \eqref{27}, to conclude the proof it is enough to show that
there exists an extension $\mb h: E_N \to [0,1]$ of the function $h:
\ms J^+ \to [0,1]$, $\ms J^+ = \cup_{x\in\bb T_L} \ms J^+_x$, such
that
\begin{equation}
\label{20}
\lim_{N\to \infty} N^{1+\alpha} \< \mb h, (-\mc L_{\complement}) \mb
h\>_{\mu_N} \;=\; 0\;.
\end{equation}

Indeed, the supremum appearing in \eqref{23} is certainly bounded
below by
\begin{equation*}
\Big\{ 2\, \< f, \mc L_{B} \mb h\>_{\mu_N} -
\< \mb h, (-\mc L_{B}) \mb h\>_{\mu_N} \Big\}
\;+\;  \Big\{ 2\, \< f, \mc L_{\complement} \mb h\>_{\mu_N} -
\< \mb h, (-\mc L_{\complement}) \mb h\>_{\mu_N} \Big\}\;.
\end{equation*}
We have shown that the first term multiplied by $N^{1+\alpha}$ is
bounded below by $(L-1)/(L \Gamma(\alpha) I_\alpha) - o_N(1)$.  On the
other hand, since $\< f, (-\mc L_{\complement}) f\>_{\mu_N} \le
D_N(f)$, by the sector condition \eqref{21}, the second term is
bounded below by
\begin{equation*}
- 4 \, L \,  D_N(f)^{1/2} \< h, (-\mc L_{\complement})
h\>_{\mu_N}^{1/2} \;-\;
\< h, (-\mc L_{\complement})\, h\>_{\mu_N}\;,
\end{equation*}
By the a-priori estimate \eqref{19} and by \eqref{20}, this expression
multiplied by $N^{1+\alpha}$ vanishes as $N\uparrow\infty$. This
proves \eqref{23}.

To conclude the proof of the proposition it remains to show the
validity of \eqref{20}. We start with a simple remark. Denote by $\ms
U_{N,m}$, $1\le m\le N$, the set of configurations in $E_N$ with at
least three sites occupied by at least $m$ particles, $\ms U_{N,m} =
\{\eta\in E_N : \exists\, x_1, x_2, x_3 , \eta_{x_i} \ge m, x_i\not =
x_j \}$. We claim that
\begin{equation}
\label{24}
\mu_N(\ms U_{N.m}) \;\le\; \frac{C_0}{m^{2(\alpha-1)}} 
\end{equation}
for some finite constant $C_0$ independent of $N$ and $m$.

To prove \eqref{24}, by symmetry, it is enough to estimate the measure
of the set $\{\eta \in E_N : \eta_0 , \eta_1 , \eta_2 \ge m\}$. By
definition of $\mu_N$ and in view of \eqref{zk}, the measure of this
set is bounded by
\begin{equation*}
C_0\, N^\alpha\, \sum_{\substack{\eta\in E_N\\ \eta_0, \eta_1, \eta_2 \ge m}}
\frac 1{a(\eta)}
\end{equation*}
for some finite constant $C_0$ independent of $N$, $m$ and whose value
may change from line to line. By definition of $\mu_N$ again, we may
rewrite the previous sum as
\begin{equation*}
C_0\, N^\alpha\, \sum_{M=3m}^{N} 
\sum_{\substack{\eta_0+\eta_1+\eta_2 = M \\ \eta_0, \eta_1, \eta_2 \ge m}}
\frac 1{a(\eta_0)a(\eta_1)a(\eta_2)}
\sum_{\eta_3+\cdots+\eta_{L-1} = N-M } \frac 1{a(\eta_3)\cdots a(\eta_{L-1})}
\end{equation*}
In view of \eqref{f03} and \eqref{zk}, the last sum is bounded by
$C_0 (N-M)^{-\alpha}$. The previous expression is thus less than or
equal to
\begin{equation*}
C_0\,  \sum_{M=3m}^{N} \frac{N^\alpha}{M^\alpha (N-M)^\alpha}
\, \mu_{3,M} (\eta_0, \eta_1, \eta_2 \ge m)\;,
\end{equation*}
where $\mu_{3,M}$ stands for the stationary measure over three sites
with $M$ particles. It remains to show that
\begin{equation*}
\mu_{3,M} (\eta_0, \eta_1, \eta_2 \ge m) \;\le\; \frac{C_0}
{m^{2(\alpha-1)}}
\end{equation*}
uniformly over $M\ge 3m$, which is elementary.

Let $\mb h$ be an extension to $E_N$ of the function $h$ defined in
\eqref{17} and assume that $\mb h$ is absolutely bounded by one. By
\eqref{24},
\begin{equation}
\label{28}
\sum_{\zeta\in \ms U_{N-1,m}} \<\mb h , (-\mc L_{\zeta})
\mb  h\>_{\mu_N} \;\le\; \frac{C_0}{ m^{2(\alpha-1)}}
\end{equation} 
for some finite constant $C_0$.  This concludes the proof of the
proposition for $\alpha > 3$ and $\ell_N \gg N^{(1+\alpha)/(2\alpha
  -2)}$. The condition $\alpha >3$ is needed to ensure that
$(1+\alpha)/(2\alpha -2)<1$. To prove the claim note that under these
hypotheses we may choose $k_N\ll \ell_N$ such that $k_N \gg
N^{(1+\alpha)/(2\alpha -2)}$. Defining $h$ by \eqref{17} on $\ms
J_0^+$ and extending its definition to the whole space $E_N$ in an
arbitrary way which keeps the function non-negative and bounded by one
and which respects the values imposed on the sets $\widehat{\ms E}_N$,
$\breve{\ms E}_N$, we obtain a function $\mb h: E_N\to [0,1]$ in $\mc
C_{1,0}(\widehat{\ms E}_N, \breve{\ms E}_N)$ such that
\begin{equation*}
\< \mb h, (-\mc L_{\complement}) \mb h\>_{\mu_N} \;\le\;
\sum_{\zeta\in \ms U_{N-1,k_N}} \<\mb h , (-\mc L_{\zeta})
\mb  h\>_{\mu_N}
\end{equation*}
because $\ms J^c_0 \subset \ms U_{N-1,k_N} \cup \widehat{\ms
  E}_{N}\cup \breve{\ms E}_{N}$. There is an abuse of notation in the
last relation but the meaning is clear. If a configuration $\zeta$
belongs to $\ms J^c_0$ and does not belong to $\ms U_{N-1,k_N}$ then
it belongs to the deep interior of one of the wells $\ms E^x_N$
where $h$ is constant and $\mc L_{\zeta} h = 0$.  It follows from the
previous estimate and from \eqref{28} that
\begin{equation*}
N^{1+\alpha} \< \mb h, (-\mc L_{\complement}) \mb h\>_{\mu_N} \;\le\;
\frac{C_0 N^{1+\alpha}}{ k_N^{2(\alpha-1)}}
\end{equation*}
for some finite constant $C_0$. By definition of the sequence $k_N$,
this expression vanishes as $N\uparrow\infty$, which proves \eqref{20}
and the proposition.

\smallskip\noindent{\bf Step 3. Filling the gap.}  Fix $0<\epsilon
<1/3$ and denote by $\ms W_\epsilon$ the set of configurations in
which at most two sites are occupied by more than $\epsilon N$
particles, $\ms W_\epsilon = \cup_{x<y} \{\eta\in E_N : \eta_z \le
\epsilon N\,,\, z\not = x,y\}$. Since $\alpha>3$, by \eqref{24} we
need only to define $\mb h$ on the set $\ms W_\epsilon$.

For $x<y$, denote by $H_{x,y}$ the functions $H:\{0, \dots, N\}\to
[0,1]$ introduced in Step 1 to define $h$ on the strips $\ms
J^+_{x,y}$. In particular, $H_{x,y}=0$ if $x$, $y\not =0$.  The
function $\mb h$ will be defined on $\ms W_\epsilon \setminus \ms J^+$
as a convex combination of these functions.  Let $\Sigma_L$ be the
simplex $\Sigma_L = \{(\theta_0, \dots , \theta_{L-1}) : \theta_i \ge
0 \,,\, \sum_i \theta_i =1\}$ and let $\mb h : \ms W_\epsilon \to
[0,1]$ be given by
\begin{equation}
\label{33}
\mb h(\eta) \;=\; \sum_{x<y} \theta_{x,y}(\eta/N) H_{x,y} (\eta_y)\;,
\end{equation}
where the function $\theta: \Sigma_L \to \Sigma_{L(L-1)/2}$ fulfills
the following conditions:
\begin{itemize}
\item Each component is Lipschitz continuous,
\item On the set $\{\eta: \eta_x \ge (1-3\epsilon)N\}$, $\sum_{y\not
    =x}\theta_{x,y} (\eta)=1$, where $\theta_{v,w} = \theta_{w,v}$,
\item On the set $\{\eta\in E_N : \eta_z \le \epsilon N \,,\, z\not =
  x,y \,,\, \eta_x, \eta_y \ge 2\epsilon N\}$, $\theta_{x,y}(\eta) = 1$.
\end{itemize}

There is a slight inaccuracy in the definition of $\mb h$ above. The
correct definition of $\mb h$ involves a convex combination of the
values of $h$ at the \emph{boundary} of $\ms J^+$. In \eqref{33}, we
didn't use the values of $h$ at the boundary of $\ms J^+$ but the
values in the deep interior of the strips $\ms J_{x,y}$ to avoid
complicated formulae. To clarify this remark, note that for $L=3$
$h_{0,1}(i,k_N+1) = H_{0,1}(i)$ but it is not true that
$h_{0,2}(k_N+1,j) = H_{0,2}(j)$, we have instead $h_{0,2}(k_N+1,j) =
H_{0,2}(j+k_N+1)$. Hence, to be absolutely rigorous, instead of
$H_{0,2}(j)$ we should have $h_{0,2}(k_N+1,j)$ in \eqref{33}.

We present the proof for $L=3$ to keep notation simple. In this case
we represent the dynamics as a random walk in the simplex $\{(i,j) :
i,j\ge 0, i+j\le N\}$, where the first coordinate stands for the
variable $\eta_1$ and the second one for the variable $\eta_2$.  We
omit from the notation the dependence on the variable $\eta_0$,
writing a function $f:E_N\to\bb R$ simply as $f(\eta_1, \eta_2)$.

We turn to the proof of condition \eqref{24}. Consider the subset $\ms
W^2_\epsilon$ of $\ms W_\epsilon$ consisting of all configurations
$\eta$ such that $\eta_2\le \epsilon N$, $\ms W^2_\epsilon =
\{\eta\in\ms W_\epsilon : \eta_2\le \epsilon N\}$. The other cases are
treated similarly. The set $\ms W^2_\epsilon$ can be decomposed in
three different regions, $\{\eta: \eta_1 \le 2\epsilon N\}$, $\{\eta:
\eta_0 \le 2\epsilon N\}$ and $\{\eta: \eta_1, \eta_0 \ge 2\epsilon
N\}$. The Dirichlet from is estimated in the same way in the first two
sets and is easier to estimate in the latter. We concentrated,
therefore, in the first set. Assume that $\eta_2\le \epsilon N$,
$\eta_1 \le 2\epsilon N$. 

For $L=3$ with the simplex representation, the Dirichlet form has
three types of terms. Jumps from $(i,j)$ to $(i+1,j)$, which
corresponds to jumps of particles from site $0$ to site $1$, jumps
from $(j,i)$ to $(j+1,i)$, and jumps from $(i+1,j)$ to $(i,j+1)$. We
estimate the contributions to the total Dirichlet form of the first
one. The second one is handled similarly and the third one can be
decomposed as a jump from $(i+1,j)$ to $(i,j)$ and then one from
$(i,j)$ to $(i,j+1)$.

On the set $\eta_2\le \epsilon N$, $\eta_1 \le 2\epsilon N$,
$\theta_{0,1} + \theta_{0,2}=1$ and the difference $\mb h(i+1,j) - \mb
h(i,j)$ can be written as
\begin{equation*}
\big[ \theta_{0,1}(i+1,j) - \theta_{0,1}(i,j) \big]\,
\big[H_{0,1}(i+1) - H_{0,2}(j)\big] \;+\; \theta_{0,1}(i,j) 
\big[H_{0,1}(i+1) - H_{0,1}(i)\big]\;.
\end{equation*}
Since $\theta$ is Lipschitz continuous and since the difference
$H_{0,1}(i+1) - H_{0,2}(j)$ can be written as $[H_{0,1}(i+1) - 1] +
[1-H_{0,2}(j)]$, the square of the previous expression is less than or
equal to
\begin{equation*}
\frac{C_0}{N^2} \big\{ [H_{0,1}(i+1) - 1]^2 + [H_{0,2}(j)-1]^2 \big\}
\;+\; C_0 \big[H_{0,1}(i+1) - H_{0,1}(i)\big]^2 \;.
\end{equation*}

The contribution to the total Dirichlet form of the jumps from $(i,j)$
to $(i+1,j)$ on the set $\eta_2\le \epsilon N$, $\eta_1 \le 2\epsilon
N$ is bounded by
\begin{equation*}
C_0\, N^\alpha \sum_{i=k_N}^{2\epsilon N} \sum_{j=k_N}^{\epsilon N} \frac
1 {a(i)a(j)a(N-i-j)} \, [\mb h(i+1,j) - \mb h(i,j)]^2 
\end{equation*}
for some finite constant $C_0$.  By the previous bound for the square
of the difference $\mb h(i+1,j) - \mb h(i,j)$, this expression is less
than or equal to the sum of three terms. We estimate the first one.
The second one is handled in a similar way and the third one is
simpler.

The first term is given by
\begin{equation*}
\frac{C_0\, N^\alpha}{N^2 k^{\alpha-1}_N} \sum_{i=k_N}^{2\epsilon N}
\frac 1 {a(i)a(N-i)} \, [H_{0,1}(i+1) - 1]^2 
\end{equation*}
Recall that $H_{0,1}(i)=1$ for $i\le \ell_N-k_N$.  Replacing $1$ by
$H_{0,1}(\ell_N-k_N)$, writing the difference $H_{0,1}(i+1) -
H_{0,1}(\ell_N-k_N)$ as a telescopic sum and applying Schwarz
inequality, we bound the last expression by
\begin{equation*}
\frac {C_0 N^\alpha}{k^{\alpha-1}_N N^2} \sum_{i=\ell_N-k_N}^{2\epsilon
  N} 
\frac {i^{\alpha+1}} {a(N-i)a(i)} \, \sum_{j=\ell_N-k_N}^{i}
[H_{0,1}(j+1) - H_{0,1}(j)]^2 \frac 1{a(j)} \;.
\end{equation*}
The factor $w^{\alpha+1}$ came from the sum $\sum_j a(j)$. It remains
to change the order of summations to bound the previous sum by
\begin{equation*}
\frac {C_0 N^\alpha}{k^{\alpha-1}_N N^2} \sum_{j=\ell_N-k_N}^{2 \epsilon
  N} \frac 1{a(j)a(N-j)} \, [H_{0,1}(j+1) - H_{0,1}(j)]^2 
\sum_{i= j}^{2\epsilon N} i \;\le\;
\frac {C_0 (\epsilon)}{k^{\alpha-1}_N N^{1+\alpha}} \;,
\end{equation*}
where the last inequality follows from \eqref{18}. If one recalls that
the Dirichlet form is multiplied by $N^{1+\alpha}$, the contribution
to the global Dirichlet from of the piece we just estimated is bounded
by $k^{-(\alpha-1)}_N$. This concludes the proof of the proposition.
\end{proof}

\section{A variational problem for the mean jump rate}
\label{sec1}

Denote by $R_N^{\ms E}(\cdot,\cdot)$ the jump rates of the trace
process $\{ \eta^{\ms E_N} (t) : t\ge 0\}$ defined in Section
\ref{sec0}. For $x \not =y$ in $\bb T_L$, let
\begin{equation}
\label{13}
r_N(\ms E^x_N ,\ms E^y_N) \;:=\; \frac{1}{\mu_N(\ms E^x_N)} 
\sum_{\substack{\eta\in \ms E^x_N \\ \xi \in \ms E^y_N }} \mu_N(\eta) 
\, R_N^{\ms E}(\eta,\xi)\;.
\end{equation}
We have shown in \cite{bl7} that in contrast with the reversible case,
there is no formula in the nonreversible context relating the mean
rates $r_N(\ms E^x_N, \ms E^y_N)$ to the capacities. The mean rates
are instead defined implicitly through a class of variational problems
examined in this section.
 
Fix $x \in \bb T_L$, let $\breve{\ms E}^x_N = \cup_{y\not = x} \ms
E^y_N$ and consider the variational problem
\begin{equation}
\label{35}
\inf_f \sup_{h} 
\Big\{ 2 \<   f \,,\, \mc L h\>_{\mu_N}  \, - 
\, \< h , (- \mc S) h\>_{\mu_N} \Big\}
\end{equation}
where the infimum is carried over all functions $f$ in $\mc
C_{1,0}(\ms E^x_N, \breve{\ms E}^x_N)$ and the supremum over all
functions $h$ in $\mc C (\ms E^x_N, \breve{\ms E}^x_N)$.

\begin{remark}
\label{s03}
We learn from the proof of Proposition \ref{s08} that in the
variational problem \eqref{35} we may restrict the supremum to
functions $f$ of the form \eqref{pf0}. More exactly, if we denote by
$a_N$ the expression in \eqref{35} and by $b_N$ the expression in
\eqref{35} when the infimum is carried over functions $f$ of the form
\eqref{pf0}, $a_N\le b_N \le a_N + c_N$, where $N^{1+\alpha} c_N$
vanishes as $N\uparrow\infty$.

We learn from the proof of Proposition \ref{s05} that we may restrict
the supremum in \eqref{35} to functions $h$ which on the strips $\ms
J^+_{y,z}$, $z\not = y \in \bb T_L$, depend on $\eta_{y+1}+\cdots +
\eta_z$, $h(\eta) = U(\eta_{y+1}+\cdots + \eta_z)$ for some $U:\bb
N\to\bb R_+$, and which are equal to $1$, $0$ on the sets $\ms E^x_N$,
$\breve{\ms E}^x_N$, respectively.
\end{remark}

For three pairwise disjoint subsets $A_1$, $A_2$, $A_3$ of $E_N$ and
three numbers $a_1$, $a_2$, $a_3$ in $[0,1]$, denote by $\bb G^{A_1,
  A_2, A_3}_{a_1, a_2, a_3}$ the functional defined on functions
$f:E_N\to\bb R$ by
\begin{equation}
\label{22}
\bb G^{A_1, A_2, A_3}_{a_1, a_2, a_3} (f) \;=\; 
\sup_{h} \Big\{ 2 \<   f \,,\, \mc L h\>_{\mu_N}  \, - 
\, \< h , (- \mc S) h\>_{\mu_N} \Big\}\;,
\end{equation}
where the supremum is carried over all functions $h: E_N\to \bb R$,
which are equal to $a_i$ on $A_i$, $1\le i\le 3$.  When we require $h$
to be only constant on the set $A_i$, we replace $a_i$ by
$*$. Moreover, when $A_1 = \ms E^x_N$, $A_2 = \ms E^y_N$, $x\not = y
\in \bb T_L$, $A_3 = \cup_{z\not = x,y} \ms E^z_N$, we denote $\bb
G^{A_1, A_2, A_3}_{*, *, 0}$, $\bb G^{A_1, A_2, A_3}_{a_1, a_2, 0}$ by
$\bb G^{x,y}$, $\bb G^{x,y}_{a_1, a_2}$, respectively.

\begin{proposition}
\label{s04}
Suppose that $\alpha>3$.  Let $\{\ell_N : N\ge 1\}$ be a sequence
satisfying \eqref{f18} and fix $x\not = y \in \bb T_L$. Then,
\begin{equation}
\label{29}
\lim_{N\to\infty} N^{1+\alpha} \inf_f 
\bb G^{x,y} (f) \;=\; \frac 1{\Gamma(\alpha) \, I_\alpha} 
\, \frac {L-2}{L-1}\; ,
\end{equation}
where the infimum is carried over all functions $f$ which are equal to
$1$ on $\ms E^x_N$, $0$ on $\cup_{z\not = x,y} \ms E^z_N$ and which
are constant on $\ms E^y_N$.
\end{proposition}

The proof of this proposition is analogous to the one of Theorem
\ref{mt1} and left to the reader. In the proof of the upper bound we
set the value $\beta = 1/(L-1)$ for $f$ on the set $\ms E^y_N$ and
define in \eqref{36} $F_y(\eta)$ by $\beta W_y(\eta/N)$ instead of
$W_y(\eta/N)$. The proof of the lower bound is carried as the proof of
Proposition \ref{s05}.

The optimal value at the set $\ms E^y_N$ for the function $f$ is
$1/(L-1)$. More precisely,
\begin{equation*}
\lim_{N\to\infty} N^{1+\alpha} \inf_f 
\bb G^{x,y} (f) \;=\; \frac 1{\Gamma(\alpha) \, I_\alpha\, L} 
\, \big\{ (L-1)(1+\gamma^2) - 2\gamma \big\}\; .
\end{equation*}
if the infimum is carried over functions $f$ whose value at $\ms
E^y_N$ is $\gamma_N\to \gamma\in [0,1]$.

If the function $h$ appearing in the variational formula \eqref{22}
does not coincide with $f$ on $\ms E^x_N$ or on $\ms E^y_N$ the left
hand side in \eqref{29} is strictly smaller than the right hand side.

\begin{lemma}
\label{s06}
Consider sequences $\{\gamma_N : N\ge 1\}$, $\{\beta_N : N\ge 1\}$
such that $(\gamma_N, \beta_N)$ does not converge to $([L-1]^{-1},
1)$. Then,
\begin{equation*}
\limsup_{N\to\infty} N^{1+\alpha} \inf_f 
\bb G^{x,y}_{\beta_N , \gamma_N} (f) \;<\; \frac 1{\Gamma(\alpha) \, I_\alpha} 
\, \frac {L-2}{L-1}\;, 
\end{equation*}
where the infimum is carried over functions $f$ which are equal to
$1$, $(L-1)^{-1}$, $0$ on $\ms E^x_N$, $\ms E^y_N$, $\cup_{z\not =
  x,y} \ms E^z_N$, respectively.
\end{lemma}

\begin{proof}
To prove the lemma we may restrict the infimum to functions $f$
which are of the form $f = F_x + (L-1)^{-1} F_y$, where $F_z$ has been
introduced in \eqref{36}. To fix ideas, consider a function $h$ whose
value at $\ms E^x_N$ is $\beta_N$ for a sequence $\beta_N$ which does
not converge to $1$. In view of the proof of Proposition \ref{s08},
where the difference $2 \< f \,,\, \mc L h\>_{\mu_N} \, - \, \< h , (-
\mc S) h\>_{\mu_N}$ is analyzed separately on each strip $\ms
I^{z,w}_N$, it is enough to show that
\begin{equation*}
\limsup_{N\to\infty} N^{1+\alpha} \sup_{h} \big\{ 2\<F_A, \mc L_{x,z} h \>_{\mu_N} 
\;-\; \<h, (-\mc L_{x,z}) h \>_{\mu_N} \big\} \;<\; 
\frac 1{\Gamma(\alpha) \, I_\alpha\, L}
\end{equation*}
for some $z\not = x$, $y$.  It will be more convenient here to define
the set $\ms I^{x,z}_{N}$ as the set of configurations of $E_N$ such
that $\eta_w\le k_N$ for all $w\not = x$, $z$, where $k_N$ is a
sequence such that $1\ll k_N\ll \ell_N$.

Repeating the computations below \eqref{04}, we write the linear term
as
\begin{equation*}
\begin{split}
& \frac{N^\alpha}{Z_N} \sum_{\xi\in \ms I^{x,z}_{N-1}} 
\,\frac 1{a(\xi)}\,  \{h (\xi+\mf d_{x+1}) - h (\xi+\mf d_{x})\}
\, \{\bb W ([\xi_x+1]/N) - \bb W (\xi_x/N)\} \\
&\qquad -\; \frac{N^\alpha}{Z_N} \sum_{\xi\in \ms I^{x,z}_{N-1}} 
\,\frac 1{a(\xi)}\, \{h (\xi+\mf d_{z+1}) - h (\xi+\mf d_{z})\}
\, \{\bb W ([\xi_z+1]/N) - \bb W (\xi_z/N)\}\;.   
\end{split}
\end{equation*}
We estimate the first term, the second one is handled similarly. We may
rewrite the first sum as
\begin{equation*}
\frac{N^\alpha}{Z_N} \sum_{k=0}^{N-1} \{\bb W ([k+1]/N) - \bb W
(k/N)\} \sum_{\xi\in B_k} \frac 1{a(\xi)}\,  
\{h (\xi+\mf d_{x+1}) - h (\xi+\mf d_{x})\} \,,
\end{equation*}
where the second sum is carried over all configurations $\xi$ in $\ms
I^{x,z}_{N-1}$ such that $\xi_x =k$. By definition of $\bb W$, we may
in fact restrict the first sum to the interval $\{\ell_N/2, \dots, N-
\ell_N\}$. Add and subtract the expression $\mf M Z_N/ N^{1+2\alpha}$
to the sum over $B_k$, where $\mf M$ is an arbitrary constant. The
telescopic sum of $\bb W ([k+1]/N) - \bb W (k/N)$ gives $\mf M /
N^{1+\alpha}$ by definition of $\bb W$. The other term is estimated
using Schwarz inequality. After performing these steps, we obtain that
the previous expression is less than or equal to
\begin{equation*}
\begin{split}
& \frac{N^\alpha \Gamma(\alpha)^{L-2}}{2 Z_N} \sum_{k=0}^{N-1} 
\frac 1{a(k) a(\bar N-k)} \{\bb W ([k+1]/N) - \bb W (k/N)\}^2 \;+\; 
\frac{\mf M}{N^{1+\alpha}} \\
&\quad +\; \frac{N^\alpha}{2 Z_N} 
\sum_{k=\ell_N/2}^{N-\ell_N} 
\frac{a(k) a(\bar N-k)}{\Gamma(\alpha)^{L-2}} 
\Big\{ \sum_{\xi\in B_k} \frac 1{a(\xi)}\,  
[h (\xi+\mf d_{x+1}) - h (\xi+\mf d_{x})]  - 
\frac{\mf M \, Z_N}{N^{1+2\alpha}} \Big\}^2 \;,
\end{split}
\end{equation*}
where $\bar N= N-1$. By definition of the function $\bb W$, the first
term is equal to $(\sigma_N/2 \Theta_\alpha) N^{-(1+\alpha)}$, where
here and below $\sigma_N$ is a sequence which converges to $1$ as
$N\uparrow\infty$ and then $\epsilon \downarrow 0$, and
$\Theta_\alpha=L I_\alpha \Gamma (\alpha)$.

Expand the square in the second line. By \eqref{zk}, the simpler
quadratic term is equal to $\sigma_N \Theta_\alpha \mf M^2/ 2
N^{1+\alpha}$. By Schwarz inequality, the second quadratic term is
bounded by
\begin{equation}
\label{37}
\frac{N^\alpha}{2 Z_N} \sum_{k=\ell_N/2}^{N-\ell_N} 
\sum_{\xi\in B_k} \frac 1{a(\xi)}\,  
[h (\xi+\mf d_{x+1}) - h (\xi+\mf d_{x})]^2 
\sum_{\xi'\in B_k} \frac{a(k) a(\bar
  N-k)}{a(\xi')\Gamma(\alpha)^{L-2}} \;\cdot
\end{equation}
In the last sum, the factor $a(k)$ in the numerator cancels with a
similar one appearing in the denominator and $a(\bar N-k)/a(\bar N-k -
\sum_{w\not = x,z} \xi_w)$ is bounded by $1+ o_N(1)$ because $\bar N -
k \gg \sum_{w\not = x,z} \xi_w$. By definition of $\Gamma(\alpha)$, we
conclude that this expression is thus bounded by
\begin{equation*}
\begin{split}
& \frac{N^\alpha (1+o_N(1))}{2 Z_N} \sum_{k=\ell_N/2}^{N-\ell_N} 
\sum_{\xi\in B_k} \frac 1{a(\xi)}\,  
[h (\xi+\mf d_{x+1}) - h (\xi+\mf d_{x})]^2 \\
& \;=\; \frac{N^\alpha (1+o_N(1))}{2 Z_N} 
\sum_{\xi\in \ms I^{x,z}_{N-1}} \frac 1{a(\xi)}\,  
[h (\xi+\mf d_{x+1}) - h (\xi+\mf d_{x})]^2\;. 
\end{split}
\end{equation*}

It remains to estimate the linear term. It is equal to
\begin{equation*}
\begin{split}
& - \frac{\mf M}{N^{1+\alpha}} \sum_{k=\ell_N/2}^{N-\ell_N} 
\frac{1}{\Gamma(\alpha)^{L-2}} 
\sum_{\zeta} \frac 1{a(\zeta)}\, [h (k+1, \zeta) - h (k,\zeta)]  
\, \Big\{\frac{a(\bar N-k)}{a(\bar N-k - |\zeta|)} - 1 \Big\} \\
& \quad - \frac{\mf M}{N^{1+\alpha}} 
\sum_{k=\ell_N/2}^{N-\ell_N}  [H(k+1) - H(k)]\;,
\end{split}
\end{equation*}
where $H(k) = \Gamma(\alpha)^{-(L-2)}\sum_{\zeta} a(\zeta)^{-1}\, h
(k, \zeta)$, the sum in $\zeta$ is carried over all configurations
$(\zeta_1, \dots, \zeta_{L-2})$ such that $\zeta_z\le k_N$,
$(k,\zeta)$ is the configuration $\xi\in \ms I^{x,z}_{N-1}$ such that
$\xi_x = k$, $\xi_y = \bar N - k - |\zeta|$, $\xi_w = \zeta_w$, $w\not
= x$, $z$, and $|\zeta| = \sum_z \zeta_z$.  By the boundary conditions
satisfied by $h$, the second sum is equal to $-\mf M \beta_N
\sigma_N/N^{1+\alpha}$. By Schwarz inequality, the first one is
bounded by
\begin{equation*}
\begin{split}
& \frac{o_N(1) N^\alpha}{Z_N} 
\sum_{\xi\in \ms I^{x,z}_{N-1}} \frac 1{a(\xi)}\,  
[h (\xi+\mf d_{x+1}) - h (\xi+\mf d_{x})]^2  \\
& +\; \frac{Z_N  \, \mf M^2}{o_N(1) N^{2+3\alpha}}
\sum_{k=\ell_N/2}^{N-\ell_N} 
\sum_{\zeta} \frac {a(k) \, a(\bar N - k - |\zeta|)}
{a(\zeta) \, \Gamma(\alpha)^{2(L-2)}}\,   
\, \Big\{\frac{a(\bar N-k)}{a(\bar N-k - |\zeta|)} - 1 \Big\}^2\;.
\end{split}
\end{equation*}
Since $|\zeta|\ll N-k$, we may choose the expression $o_N(1)$
appearing in this formula to decrease slowly enough for $\{ [a(\bar
N-k)/a(\bar N-k - |\zeta|)] - 1 \}^2 / o_N(1)$ to vanish as
$N\uparrow\infty$. With that choice the second line of the previous
formula becomes bounded by $o_N(1) \mf M^2 \Theta_\alpha
N^{-(1+\alpha)} = o_N(1) \mf M^2 N^{-(1+\alpha)}$.

Up to this point we showed that the first part of the linear term
$2\<F_A, \mc L_{x,z} h \>_{\mu_N}$ is bounded by
\begin{equation*}
\begin{split}
& \frac 1{N^{1+\alpha}} \Big\{ 
\frac {\sigma_N}{2 \Theta_\alpha} 
\;+\; \mf M  (1-\beta_N \sigma_N) \;+\;
\frac{\Theta_\alpha \sigma_N}2 \mf M^2 \Big\}\\ 
&\quad +\; \frac{N^\alpha (1+o_N(1))}{2 Z_N} 
\sum_{\xi\in \ms I^{x,z}_{N-1}} \frac 1{a(\xi)}\,  
[h (\xi+\mf d_{x+1}) - h (\xi+\mf d_{x})]^2 \;.
\end{split}
\end{equation*}
In the formula above \eqref{37}, if we apply the inequality $2ab \le R
a^2 + R^{-1} b^2$, $R>0$, instead of the inequality $2ab \le a^2 +
b^2$, we may replace the term $ 1 + o_N(1)$ appearing in the second
line by $1$, without changing the first line since there is already
$\sigma_N$ multiplying $(2 \Theta_\alpha)^{-1}$.

Estimating the second piece of $2\<F_A, \mc L_{x,z} h \>_{\mu_N}$ in
the same way, we get that $2\<F_A, \mc L_{x,z} h \>_{\mu_N}$ is
bounded by
\begin{equation*}
\begin{split}
& \frac 1{N^{1+\alpha}} \Big\{ 
\frac {\sigma_N}{\Theta_\alpha} 
\;+\; 2 \mf M  (1-\beta_N \sigma_N) \;+\;
\Theta_\alpha \sigma_N\, \mf M^2 \Big\}\\ 
&\quad +\; \frac{N^\alpha}{2 Z_N} 
\sum_{\xi\in \ms I^{x,z}_{N-1}} \frac 1{a(\xi)}\,  
[h (\xi+\mf d_{x+1}) - h (\xi+\mf d_{x})]^2  \\
& \qquad +\; \frac{N^\alpha}{2 Z_N} 
\sum_{\xi\in \ms I^{x,z}_{N-1}} \frac 1{a(\xi)}\,  
[h (\xi+\mf d_{z+1}) - h (\xi+\mf d_{z})]^2 \;.
\end{split}
\end{equation*}
for any constant $\mf M$. In view of \eqref{03}, the sum of the last
two lines is bounded by the Dirichlet form $\<h, (-\mc L_{x,z}) h
\>_{\mu_N}$. Replacing $\mf M$ by its optimal value
$(\beta_N \sigma_N-1)/\sigma_N \Theta_\alpha$, we obtain that
\begin{equation*}
N^{1+\alpha} \sup_h \big\{ 2\<F_A, \mc L_{x,z} h \>_{\mu_N} 
\;-\; \<h, (-\mc L_{x,z}) h \>_{\mu_N} \big\} \;\le\; 
\frac {\sigma_N}{\Theta_\alpha}  \;-\; 
\frac{(1-\beta_N \sigma_N)^2}{\Theta_\alpha \sigma_N}\;\cdot
\end{equation*}
This proves the lemma.
\end{proof}

It follows from Lemma \ref{s06} that the value at $\ms E^y_N$ of the
optimal function $h$ for the variational problem \eqref{29},
\eqref{22} is asymptotically equal to $1/(L-1)$. Hence, by
\cite[Proposition 3.2]{bl7},
\begin{equation*}
\lim_{N\to\infty} \frac{r_N(\ms E^y_N,\ms E^x_N)}{r_N(\ms E^y_N,
\breve{\ms E}^y_N)} \;=\; \frac 1{L-1}\;.
\end{equation*}
By equation (5.8) in \cite{bl7}, $r_N(\ms E^y_N, \breve{\ms E}^y_N) =
\Cap (\ms E^y_N, \breve{\ms E}^y_N)/\mu_N (\ms E^y_N)$. Hence, in view
of Theorem \ref{mt1} and the fact that $\lim_{N\to\infty} \mu_N (\ms
E^y_N) = L^{-1}$, $N^{1+\alpha}r_N(\ms E^y_N, \breve{\ms E}^y_N)$
converges to $(L-1)/[\Gamma (\alpha) I_\alpha]$ so that.
\begin{equation}
\label{31}
\lim_{N\to\infty} N^{1+\alpha} \, r_N(\ms E^y_N,\ms E^x_N) \;=\; 
\frac 1{\Gamma (\alpha) I_\alpha}\;\cdot
\end{equation}

\section{Proof of Theorem \ref{mt2}}
\label{sec5}

In \cite{bl2}, we reduced the proof of the metastability of
\emph{reversible} Markov processes on countable sets to the
verification of three conditions, denoted by {\bf (H0)}, {\bf (H1)}
and {\bf (H2)}. The same result holds for general Markov processes on
countable state spaces \cite{bl7}.

Recall the definition of the set $\Delta_N$ introduced in
\eqref{f15}. Condition {\bf (H2)}, which requires that
\begin{equation}
\tag*{\bf (H2)} 
\lim_{N\to\infty} \frac{ \mu_N(\Delta_N) }{ \mu_N(\ms E^x_N)} 
\;=\;0 \;,\quad \forall x\in \bb T_L\;,
\end{equation}
follows from the fact that $\mu_N(\ms E^x_N)$ converges to $1/L$ for
every $x\in \bb T_L$. This last assertion is a consequence of the
symmetry of the sets $\ms E^x_N$ and of equation (3.2) in \cite{bl3}.

For each $x\in \bb T_L$, let $\xi^x_N\in E_N$ be the configuration
with $N$ particles at $x$ and let $\breve{\ms E}^x_N$ represent the
set $\ms E_N(\bb T_L\setminus \{x\})$. The second condition requires
that 
\begin{equation*}
\tag*{\bf (H1)}
\lim_{N\to\infty}\; \sup_{\eta\in \ms E^x_N} 
\frac{\Cap_N(\ms E^x_N, \breve{\ms E}^x_N)}
{\Cap_N( \eta , \xi^x_N) }\; 
=\; 0 \;, \quad \forall x\in \bb T_L \;.
\end{equation*}

Recall from Section \ref{sec2} that we denote by $\Cap^s_N$ the
capacity with respect to the reversible zero range process.
By Lemma \ref{s02}, the previous supremum is bounded
\begin{equation}
\label{32}
\sup_{\eta\in \ms E^x_N} 
\frac{4\, L^2\, \Cap^s_N(\ms E^x_N, \breve{\ms E}^x_N)}
{\Cap^s_N( \{\eta\} , \{\xi^x_N\}) }\;\cdot
\end{equation}
We have shown in \cite[Section 6]{bl3} that $\Cap^s_N( \eta ,
\xi^x_N) \ge C_0 \ell_N^{-\{(L-1)\alpha +1\}}$ for some positive
constant $C_0$ independent of $N$. A simple upper bound is given by
the following argument. Let $\eta$ be a configuration with $O(\ell_N)$
particles at each site $y\not = x$. By definition,
\begin{equation*}
\Cap^s_N( \eta , \xi^x_N) \;=\; \mu_N(\eta) \mb P^s_\eta \big[
H_{\xi^x_N} < H^+_{\eta} \big] \;\le\; \mu_N(\eta) \;\le\;
C_0 \ell_N^{-(L-1)\alpha}\;.
\end{equation*}
In this formula, $\mb P^s_\eta$ stands for the probability on the path
space $D(\bb R_+, E)$ induced by the Markov process with generator
$\mc S$ starting from $\eta$. This computations shows that
$\ell_N^{-(L-1)\alpha}$ should be the correct order and that the lower
bounded obtained in \cite{bl3} is not far away from the correct
order. In any cases, by \eqref{34} if $\ell_N^{(L-1)\alpha +1}\ll
N^{1+\alpha}$, \eqref{32} vanishes as $N\uparrow\infty$, proving {\bf
  (H1)}.

Finally, condition {\bf (H0)} imposes the average rates $r_N(\ms E^x_N
, \ms E^y_N)$ of the trace process defined in \eqref{13} to converge:
\begin{equation*}
\tag*{\bf (H0)}
\lim_{N\to\infty} N^{1+\alpha} \, r_N(\ms E^x_N , \ms E^y_N) \;=\; 
r(x,y)  \;, 
\quad \forall x,y\in \bb T_L\,,\;x\not = y\;.
\end{equation*}
This property has been proved in \eqref{31}.

\smallskip\noindent{\bf Acknowledgments.} The author wishes to thank
J. Beltr\'an and A. Gaudilli\`ere for fruitful discussions on
metastability and their remarks on a preliminary version of this work.

\end{document}